\documentclass[a4paper,11pt, reqno]{amsart}
\usepackage{etex}
\usepackage{a4wide}
\usepackage{enumerate} 
\usepackage[shortlabels]{enumitem}
\usepackage{multirow}
\usepackage{dcpic,pictex}
\usepackage{color}
\usepackage{float}
\usepackage{bbm}
\usepackage{aliascnt}
\usepackage[utf8]{inputenc}
\usepackage[T1]{fontenc}
\usepackage[bookmarks=true,pdfstartview=FitH, pdfborder={0 0 0}, colorlinks=true,citecolor=red, linkcolor=blue]{hyperref}
\usepackage{todonotes}
\usepackage{doi}
 \usepackage{amsaddr}
\usepackage{MnSymbol}

\usepackage{eucal}

\usepackage{tikz}
\usetikzlibrary{arrows}
\usepackage{forest}
\usepackage{tkz-graph}
\usepackage{subcaption}

\numberwithin{equation}{section}

\advance\textheight by 2cm

\theoremstyle{plain}
\newtheorem{thm}{Theorem}[section]
\newaliascnt{cor}{thm} 
\newaliascnt{prop}{thm}
\newaliascnt{lem}{thm}
\newtheorem{cor}[cor]{Corollary}

\newtheorem{lem}[lem]{Lemma}
\aliascntresetthe{cor}
\aliascntresetthe{prop}
\aliascntresetthe{lem} 
\theoremstyle{definition}
\newaliascnt{defn}{thm}
\newaliascnt{asu}{thm}

\newtheorem{asu}[asu]{Assumption}
\aliascntresetthe{defn}
\aliascntresetthe{asu}
\theoremstyle{remark}
\newaliascnt{rem}{thm}
\newaliascnt{exa}{thm}
\newaliascnt{sett}{thm}
\newtheorem{rem}[rem]{Remark}
\newtheorem{exa}[exa]{Example}

\aliascntresetthe{rem}
\aliascntresetthe{exa}
\aliascntresetthe{sett}
\newenvironment{psmallmatrix}
{\left(\begin{smallmatrix}}
{\end{smallmatrix}\right)}

\newcommand{\RR}{\mathbb{R}}
\newcommand{\CC}{\mathbb{C}}
\newcommand{\NN}{\mathbb{N}}

\newcommand{\sbb}{\mathbbm{s}}
\newcommand{\eps}{\varepsilon}

\newcommand{\sL}{\mathcal{L}}
\newcommand{\sA}{\mathcal{A}}
\newcommand{\sB}{\mathcal{B}}
\newcommand{\sC}{\mathcal{C}}
\newcommand{\sK}{\mathcal{K}}
\newcommand{\sP}{\mathcal{P}}
\newcommand{\sF}{\mathcal{F}}
\newcommand{\sQ}{\mathcal{Q}}
\newcommand{\sR}{\mathcal{R}}
\newcommand{\sRtn}{\mathcal{R}_{\tn}}
\newcommand{\sX}{\mathcal{X}}
\newcommand{\sx}{{\scriptstyle\mathcal{X}}}
\newcommand{\su}{{\scriptstyle\mathcal{U}}}
\newcommand{\sv}{{\scriptstyle\mathcal{V}}}
\newcommand{\ssv}{{\scriptscriptstyle\mathcal{V}}}
\newcommand{\sG}{\mathcal{G}}
\newcommand{\sV}{\mathcal{V}}
\newcommand{\sS}{\mathcal{S}}
\newcommand{\sT}{\mathcal{T}}

\newcommand{\rL}{\mathrm{L}}
\newcommand{\rM}{\mathrm{M}}
\newcommand{\rC}{\mathrm{C}}
\newcommand{\rW}{\mathrm{W}}

\newcommand{\p}{{\raisebox{1.3pt}{{$\scriptscriptstyle\bullet$}}}}
\newcommand{\ps}{{\raisebox{0.3pt}{{$\scriptscriptstyle\bullet$}}}}

\newcommand{\RRp}{{\RR_+}}
\newcommand{\LpneCm}{\mathchoice{\rL^p\bigl([0,1],\CC^m\bigr)}{\rL^p([0,1],\CC^m)}{\rL^p([0,1],\CC^m)}{\rL^p([0,1],\CC^m)}}

\newcommand{\CneCm}{\mathchoice{{\rC\bigl([0,1],\CC^m\bigr)}}{{\rC([0,1],\CC^m)}}{{\rC([0,1],\CC^m)}}{{\rC([0,1],\CC^m)}}}
\newcommand{\WepneCm}{{\rW^{1,p}([0,1],\CC^m)}}

\newcommand{\WzpneCm}{{\rW^{2,p}([0,1],\CC^m)}}

\newcommand{\Yl}{{Y^\bot}}

\newcommand{\h}{h}

\newcommand{\Tt}{(T(t))_{t\ge0}}

\newcommand{\sTt}{(\sT(t))_{t\ge0}}

\newcommand{\dt}{\,dt}
\newcommand{\dr}{\,dr}
\newcommand{\ds}{\,ds}
\newcommand{\ddss}{{\frac{d^2}{ds^2}\,}}
\newcommand{\dds}{\frac{d}{ds}}
\newcommand{\tddss}{\textstyle{\frac{d^2}{ds^2}\,}}
\newcommand{\tdds}{\textstyle\frac{d}{ds}}
\newcommand{\diag}{\operatorname{diag}}
\newcommand{\rg}{\operatorname{rg}}

\newcommand{\inc}{\overset{\text{c}\;}{\hookrightarrow}}

\newcommand{\Xme}{X_{-1}}
\newcommand{\Ame}{A_{-1}}

\renewcommand{\r}{\right}
\renewcommand{\l}{\left}

\newcommand{\Gt}{{\tilde G}}
\newcommand{\Wb}{{\bar W}}
\newcommand{\Id}{Id}
\newcommand{\sBt}{\sB_t}

\newcommand{\Bttn}{(\sB_t)_{t\in[0,\tn]}}

\newcommand{\tn}{{t_0}}
\newcommand{\dX}{{\partial X}}
\newcommand{\phib}{{\bar\varphi^i}}

\newcommand{\cb}{{\bar c}}
\newcommand{\cbi}{{\cb}^{-1}}
\newcommand{\Cb}{{\bar C}}

\newcommand{\uh}{{\hat u}}
\newcommand{\vh}{{\hat v}}
\newcommand{\wh}{{\hat w}}
\newcommand{\fh}{{\hat f}}
\newcommand{\Phin}{{\Phi_0}}
\newcommand{\Phie}{{\Phi_1}}
\newcommand{\Phix}{{\mathbf{\Phi}}}
\newcommand{\Phieb}{{\bar\Phi_1}}
\newcommand{\Rt}{R_t}
\newcommand{\St}{S_t}
\renewcommand{\theta}{\vartheta}
\newcommand{\eins}{\mathbbm{1}}

\newcommand{\Beq}{\bar B}
\renewcommand{\u}{q}
\newcommand{\mv}{\mathsf{v}}
\newcommand{\me}{\mathsf{e}}

\newcommand{\sBte}{\sB_t^e}
\newcommand{\sBti}{\sB_t^i}
\newcommand{\mei}{\me^i}
\newcommand{\mee}{\me^e}
\newcommand{\Xe}{X^e}
\newcommand{\Xin}{X^i}
\newcommand{\sXe}{\sX^e}
\newcommand{\sXi}{\sX^i}
\newcommand{\LpRpCl}{\rL^p(\RRp,\CC^{\ell})}
\newcommand{\CnRpCl}{\rC_0(\RRp,\CC^{\ell})}
\newcommand{\WepRpCl}{\rW^{1,p}(\RRp,\CC^{\ell})}
\newcommand{\WzpRpCl}{\rW^{2,p}(\RRp,\CC^{\ell})}
\newcommand{\Tle}{T_{l}^e}
\newcommand{\Tlet}{(\Tle(t))_{t\ge0}}
\newcommand{\Tre}{T_r^e}
\newcommand{\Tret}{(\Tre(t))_{t\ge0}}
\newcommand{\Tli}{T_{l}^i}
\newcommand{\Tlit}{(\Tli(t))_{t\ge0}}
\newcommand{\Tri}{T_r^i}
\newcommand{\Trit}{(\Tri(t))_{t\ge0}}
\newcommand{\Qt}{Q_t}
\newcommand{\sTe}{\sT^e}
\newcommand{\sTet}{(\sTe(t))_{t\ge0}}
\newcommand{\sTi}{\sT^i}
\newcommand{\sTit}{(\sTi(t))_{t\ge0}}
\newcommand{\Phine}{\Phi_0^e}
\newcommand{\Phini}{\Phi_0^i}
\newcommand{\Phiee}{\Phi_1^e}
\newcommand{\Phiei}{\Phi_1^i}
\renewcommand{\ae}{\lambda^e}
\newcommand{\ai}{\lambda^i}
\newcommand{\qe}{q^e}
\newcommand{\qi}{q^i}
\newcommand{\ce}{\mu^e}
\newcommand{\ci}{\mu^i}
\newcommand{\aij}{\ai_j}
\newcommand{\aek}{\ae_k}
\newcommand{\cij}{\ci_j}
\newcommand{\cek}{\ce_k}
\newcommand{\phiek}{\varphi^e_k}
\newcommand{\phiij}{\varphi^i_j}

\newcommand{\fe}{f^e}
\newcommand{\fin}{f^i}
\newcommand{\ue}{u^e}
\newcommand{\ui}{u^i}

\newcommand{\ueh}{\uh^e}
\newcommand{\uih}{\uh^i}

\newcommand{\Phiebe}{\Phieb^e}
\newcommand{\Phiebi}{\Phieb^i}
\newcommand{\Dme}{D_m^e}
\newcommand{\Dmi}{D_m^i}
\newcommand{\Dne}{D_0^e}
\newcommand{\Dni}{D_0^i}
\newcommand{\Dei}{D_1^i}
\newcommand{\sAme}{\sA_m^e}
\newcommand{\sAmi}{\sA_m^i}
\newcommand{\sAe}{\sA^e}
\newcommand{\sAi}{\sA^i}

\newcommand{\sLe}{\sL^e}
\newcommand{\sLi}{\sL^i}
\newcommand{\sXr}{\tilde\sX}
\newcommand{\sGr}{\tilde\sG}
\newcommand{\loc}{\mathrm{loc}}
\newcommand{\Rne}{R_0^e}
\newcommand{\Rni}{R_0^i}
\newcommand{\Rei}{R_1^i}
\newcommand{\Une}{U_0^e}
\newcommand{\Uni}{U_0^i}
\newcommand{\Uei}{U_1^i}
\newcommand{\Vne}{V_0^e}
\newcommand{\Vni}{V_0^i}
\newcommand{\Vei}{V_1^i}
\newcommand{\Wne}{W_0^e}
\newcommand{\Wni}{W_0^i}
\newcommand{\Wei}{W_1^i}
\newcommand{\Wneb}{\Wb_0^e}
\newcommand{\Wnib}{\Wb_0^i}
\newcommand{\Weib}{\Wb_1^i}
\newcommand{\Be}{B^e}
\newcommand{\Bi}{B^i}
\newcommand{\Te}{T^e}
\newcommand{\Ti}{T^i}
\newcommand{\R}{\Gamma}
\newcommand{\Jp}{{J_{\varphi^e}}}
\newcommand{\Jpb}{{J_{\phib}}}
\newcommand{\Jpi}{{J_{\varphi^e}^{-1}}}
\newcommand{\Jpbi}{{J_{\phib}^{-1}}}
\newcommand{\IId}{\diag(\Id,\Id)} 

\floatstyle{plain}
\newfloat{diagram}{thp}{lop}
\floatname{diagram}{Diagram}

\title[Waves and Diffusion on Metric Graphs]
{Waves and Diffusion on Metric Graphs with General Vertex Conditions}

\author{Klaus-Jochen Engel}
\address{University of L'Aquila, Department of Information Engineering, Computer Science and Mathematics, Via Vetoio, Coppito, I-67100 L'Aquila (AQ), Italy}

\author{Marjeta Kramar Fijav\v{z}}
\address{University of Ljubljana, Faculty of Civil and Geodetic Engineering, Jamova 2, SI-1000 Ljubljana, Slovenia / Institute of Mathematics, Physics, and Mechanics, Jadranska 19, SI-1000 Ljubljana, Slovenia}


\begin{document}

\begin{abstract}
We prove well-posedness for general linear wave- and diffusion equations on compact or non-compact metric graphs allowing various conditions in the vertices. More precisely, using the theory of strongly continuous operator semigroups we show that a large class of (not necessarily self-adjoint) second order differential operators with general (possibly non-local) boundary conditions generate cosine families, hence also analytic semigroups, on ${\mathrm{L}}^p({\mathbb{R}_+},{\mathbb{C}}^{\ell})\times{\mathrm{L}}^p([0,1],{\mathbb{C}}^m)$ for $1\le p<+\infty$.
\end{abstract}

\maketitle

\section{Introduction}

It is well-known (for details see \cite[Sect.~II.6]{EN:00} and \cite[Sect.~3.14]{ABHN:11}) that first and second order abstract Cauchy problems of the form
\begin{equation}\label{eq:acp}
\text{(ACP$_1$)}\quad
\begin{cases}
\dot x(t)= G x(t),&t\ge0,\\
x(0)=x_0,
\end{cases}\qquad \text{ and }\qquad
\text{(ACP$_2$)}\quad
\begin{cases}
\ddot x(t)= G x(t),&t\ge0,\\
x(0)=x_0,&\\
\dot x(0) = x_1,
\end{cases}
\end{equation}
for a linear (in general unbounded) operator $G:D(G)\subset X\to X$ on a Banach space $X$ are well-posed if and only if $G$ generates a strongly continuous semigroup and a cosine family on $X$, respectively. It follows by \cite[Thm.~3.14.17]{ABHN:11} that generators of cosine families generate analytic semigroups of angle $\frac\pi2$. Hence, well-posedness of (ACP$_2$) always implies well-posedness of  (ACP$_1$).

\smallskip
In this paper we are concerned with such Cauchy problems for second order elliptic differential operators $G$ acting on spaces of $\rL^p$-functions defined on a finite union of intervals. Operators of this type appear, e.g., in the modeling of diffusion- and wave equations on metric graphs. In this case the intervals represent the edges of the graph while its structure is encoded in the boundary conditions appearing in the domain $D(G)$ of $G$.
In the simplest case we can take $X=(\rL^p[0,1])^m=\LpneCm$ and
\begin{equation}\label{eq:example-intro}
\begin{aligned}
G&=\lambda(\p)\cdot\tddss,\\ 
D(G)&=\l\{f\in\WzpneCm\,\bigg|\,
\begin{aligned}
&V_0f(0) + V_1f(1)= 0\\[-3pt]
&W_0f'(0) - W_1f'(1)  +U_1f(1) + U_0f(0)=0
\end{aligned}
\r\},
\end{aligned}
\end{equation}
where $\lambda(s)=\diag(\lambda_j(s))_{j=1}^m$
for positive, Lipschitz continuous ``diffusion'' coefficients $\lambda_j(\p)$ and suitable  ``boundary'' matrices 
$V_0, V_1\in\rM_{k_0\times m}(\CC)$ and $U_0, U_1, W_0, W_1\in\rM_{k_1\times m}(\CC)$, for $k_0,k_1\in\NN$ satisfying $k_0+k_1=2m$. 

\smallskip
Our main result, \autoref{thm:gen-Lp}, gives for such operators a condition 
implying the generation of a cosine family, hence the well-posedness of \eqref{eq:acp}. For example, by  \autoref{cor:gen-V-W}, the operator $G$ in \eqref{eq:example-intro} generates a cosine family if
\begin{equation}\label{eq:det=/0}
\det
\begin{pmatrix}
V_1&V_0\\
W_1\cdot \mu(1)^{-1}&W_0\cdot \mu(0)^{-1}
\end{pmatrix}
\ne0,
\end{equation}
where $\mu(s):=\sqrt{\lambda(s)}=\diag(\sqrt{\lambda_j(s)})_{j=1}^m$. In particular, \eqref{eq:det=/0} implies that for $G$ as in \eqref{eq:example-intro} both Cauchy problems in \eqref{eq:acp}  are well-posed.

\goodbreak
\smallskip
Motivated by different problems from physics, chemistry, biology, and engineering, the study of dynamical processes on metric graphs (also called networks or one-dimensional ramified spaces) has received much attention in the last decades.
Diffusion equations on networks were first considered in the 1980s, the earliest references include 
\cite{Lum:80,Roth:84,Nic:85,Bel:85}. Since then many authors used functional analytic methods to treat such problems, 
we only mention \cite{GOO:93,Cat:97,ABN:01,KMS:07,BFN:16, 
Mug:14}.
The study of wave equation on networks was initiated about at the same time by \cite{Ali:84,Ali:94}, see also \cite{LLS:94, Ku:02, CF:03, GZM:05, DZ:06, KMS:07, Klo:12,Jac:15}. 

Almost simultaneously, another community of theoretical physicists was mainly interested in the Schrödinger equation on a network structure (calling it a quantum graph), see \cite{Ex:90, KS:97, KS:99, Ku:08, BK13, SSJW:15}.
They also considered so-called non-compact graphs,  where some edges are allowed to be infinite. 

All these problems were initially treated in a $L^2$-setting using Hilbert-spaces techniques. 
Then interpolation was used  to generalize the results to $L^p$-spaces. Typically, in this context only self-adjoint operators are considered.

\smallskip
On the contrary, we use methods form the theory of operator semigroups  and work on $L^p$-spaces directly. The novelty of our approach is manifold. In fact, it allows us to 
\begin{itemize}
\item study non-self-adjoint generators $G$, 
\item treat very general (also non-diagonal) ``diffusion coefficient matrices'' $a(\p,\p)$, cf. \eqref{eq:rep-a(s)},
\item treat very general boundary conditions of the form
\[\Phin f=0,\quad\Phie( f'+ B f)=0,\]
for appropriate boundary functionals $\Phin, \Phie$ and a bounded operator $B$,  cf.~\eqref{eq:def-G-general},
\item consider state spaces $X=\LpRpCl\times\LpneCm$ with application to non-compact graphs,
\item treat all cases for $p\in[1,+\infty)$ simultaneously without using interpolation arguments,
\item explicitly compute the phase space  $\ker(\Phin)\times X$ of $G$, cf.~\autoref{thm:gen-Lp}.
\end{itemize}

Our reasoning is based on a recent result for boundary perturbations of domains of generators developed in \cite{ABE:13} (which we recall in \autoref{thm:pert-bc}) and the fact that squares of group generators generate cosine families, cf. \cite[Expl.~3.14.15]{ABHN:11}. Roughly speaking, we start from a simple first-order differential operator $\sA$ generating a semigroup. Then we perturb its domain to obtain $\sG$ whose square is closely related to $G$. Moreover, since we arrange $\sG$ to be similar to $-\sG$, it automatically generates a group. Hence, $\sG^2$ and consequently also $G$ generate cosine families. To obtain our main theorem in its most general form we use similarity transformations and bounded perturbations.
In this way we are able to generalize the boundary conditions for non-self adjoint and non-compact graphs given in \cite{KPS:08,HKS:15}, see \autoref{ex-matrices},  as well as the general boundary conditions in terms of ``boundary subspaces'' presented in \cite[Sect.~6.5]{Mug:14}, see \autoref{exa:BC-Y}. We can also treat different non-local boundary conditions (for example those studied in \cite{MN:14}, see \autoref{ex-nonlocal-interval}). 

\smallskip
This paper is organized as follows. In \autoref{sec:Lp-gen} we introduce our setup, state and prove the main generation result (\autoref{thm:gen-Lp}) and apply it to two important classes of boundary conditions (\autoref{cor:gen-V-W} and \autoref{cor:net-Y}). This facilitates the verification of the  generation conditions \eqref{eq:cond-det} and \eqref{eq:intersec-0} used in \autoref{sec:W&D} to show well-posedness of diffusion- and wave equations  on (possibly non-compact) graphs for a wide variety of boundary conditions.  In the appendix we 
recall a perturbation result from \cite{ABE:13} which is the main tool for our approach. 

\smallskip
Our notation closely follows \cite{EN:00}.

\section{Generation of cosine families}
\label{sec:Lp-gen}

\subsection{The setup}
Throughout this section we make the following assumptions. Although the results presented here are abstract, the terminology already suggests that our main motivation arises from the study of dynamical processes on (possibly non-compact) metric graphs. In the sequel we use the notation $\RRp:=[0,+\infty)$. 

\goodbreak

\begin{asu}\label{asu:s-asu-Lp} 
Consider for some fixed $p\in[1,\infty)$, $\ell\in\NN_0$, and $m\in\NN_0$ satisfying $\ell+m>0$
\begin{enumerate}[(i)]
\item the space $\Xe:=\LpRpCl$ of functions on $\ell$ \emph{``external edges''},
\item the space $\Xin:=\LpneCm$ of functions on $m$ \emph{``internal edges''},
\item the \emph{``state space''} $X:=\Xe\times\Xin$ of functions on all $\ell+m$ edges,
\item two \emph{``boundary spaces''} $Y_0,Y_1\subseteq\CC^{\ell+2m}$ satisfying $Y_0\oplus Y_1=\CC^{\ell+2m}=:Y$,
\item  two \emph{``boundary functionals''}  $\Phin=(\Phine,\Phini)\in\sL(\CnRpCl\times\CneCm,Y_0)$ and $\Phie=(\Phiee,\Phiei)\in\sL(\CnRpCl\times\CneCm,Y_1)$ (to determine the boundary conditions),
\item a \emph{``boundary operator''} $B\in\sL(X)$ (appearing in the boundary conditions) leaving $\WepRpCl\times\WepneCm$ invariant,
\item  a Lipschitz continuous function%
\footnote{This implies that the multiplication operator induced by $a'(\p,\p)$ is bounded which is needed in our approach.}
$a(\p,\p)\colon\RRp\times [0,1] \to\rM_{\ell+m}(\CC)$, the so-called  
\emph{``diffusion coefficients matrix''} of the form
\begin{equation}\label{eq:rep-a(s)}
\begin{aligned}
a(\p,\p) &:=
\begin{psmallmatrix}
\qe(\ps)&0\\0&\qi(\ps)
\end{psmallmatrix}
\cdot
\begin{psmallmatrix}
\ae(\ps)&0\\0&\ai(\ps)
\end{psmallmatrix}
\cdot
\begin{psmallmatrix}
\qe(\ps)&0\\0&\qi(\ps)
\end{psmallmatrix}^{-1}\\
&:=q(\p,\p)\cdot\lambda(\p,\p)\cdot q(\p,\p)^{-1},
\end{aligned}
\end{equation}
where $\qe(\p)\in\rL^{\infty}(\RRp,\rM_{\ell}(\CC))$ and $\qi(\p)\in\rL^{\infty}([0,1],\rM_m(\CC))$ are Lipschitz continuous (pointwise) invertible with bounded inverses, and
\begin{alignat*}{3}
&\ae(s)=\diag\bigl(\aek(s)\bigr)_{k=1}^{\ell}\in\rM_{\ell}(\CC),&\quad&s\in\RRp,\\
&\ai(s)=\diag\bigl(\aij(s)\bigr)_{j=1}^m\in\rM_m(\CC),&\quad&s\in[0,1],
\end{alignat*}
with entries satisfying for some $\eps>0$
\begin{alignat}{3}\label{eq:aek-beschr}
\eps&<\aek(s)<\eps^{-1}&\quad&\text{for all }s\in\RRp,\ k=1,\ldots,\ell,
\\ \notag
0&<\aij(s)&\quad&\text{for all }s\in[0,1],\ j=1,\ldots,m.
\end{alignat}
\end{enumerate}
\end{asu}

Note that (vii) implies that $q(\p,\p)^{-1}$, and
the functions $\aek(\p)\in\rC(\RRp)$, $\aij(\p)\in\rC[0,1]$ for $k=1,\ldots,\ell$, $j=1,\ldots,m$, are all Lipschitz continuous.

\smallbreak
Using that $\WepRpCl\subset\CnRpCl$ and $\WepneCm\subset\CneCm$  (see \cite[Sect.8.2]{Bre:11}) we then define on $X=\LpRpCl\times\LpneCm$ the operator%
\begin{equation}\label{eq:def-G-general}
\begin{aligned}
G&:=a(\p,\p)\cdot\tddss,\\
D(G)&:=\bigl\{f\in\WzpRpCl\times\WzpneCm:\Phin f=0,\;\Phie( f'+ B f)=0\bigr\}.
\end{aligned}
\end{equation}
Our aim is to give conditions on the boundary functionals $\Phin,\,\Phie$ implying that $G$ generates a cosine family on $X$, hence by \cite[Thm.~3.14.17]{ABHN:11} also an analytic semigroup of angle $\frac\pi2$. As we will see in \autoref{thm:gen-Lp} below this can be achieved (independently on $B$) through an  invertibility condition on the operator $\sRtn$ defined in \eqref{eq:def-sR} below.

\smallbreak
We note that we do not consider the case $p=\infty$. In fact, this would yield a non-densely defined operator $G$ which cannot be a generator. More generally, it is well-known that on $\rL^\infty$-spaces strongly continuous semigroups are uniformly continuous, i.e., have a bounded generator. Hence, an operator $G\subset a(\p,\p)\cdot\ddss$ will never, independently on the domain, generate a $C_0$-semigroup on $\rL^\infty(\RRp,\CC^l)\times\rL^\infty([0,1],\CC^m)$.

\smallbreak
In order to state our main result rigorously we need some more notations. 
For $1\le k\le \ell$ we define
\begin{align}
\notag
\cek(\p)&:=\sqrt{\ae_k(\p)}\in\rC(\RRp)\quad\text{and set}\\
\label{eq:def-ce}
\ce(\p)&:=\diag\bigl(\ce_k(\p)\bigr)_{k=1}^{\ell}\in\rC\bigl(\RRp,\rM_{\ell}(\CC)\bigr).
\end{align}
Moreover, we put 
\begin{equation}\label{eq:def-phik}
\phiek(s):=\int_0^s\frac{dr}{\ce_k(r)},\ s\in\RRp.
\end{equation}
Then it follows from \eqref{eq:aek-beschr} that all $\phiek:\RRp\to\RRp$ are Lipschitz continuous,  surjective and strictly increasing, hence invertible with Lipschitz continuous inverses $(\phiek)^{-1}:\RRp\to\RRp$.

\smallskip
Similarly, for $1\le j\le m$ we define
\begin{align}
\notag
\cij(\p)&:=\sqrt{\aij(\p)}\in\rC[0,1]\quad\text{and set}\\
\label{eq:def-ci}
\ci(\p)&:=\diag\bigl(\cij(\p)\bigr)_{j=1}^m\in\rC\bigl([0,1],\rM_m(\CC)\bigr).
\end{align}
Furthermore, we consider
\begin{alignat}{5}\label{eq:def-varphi-j}
&\phiij(s):=\int_0^s\frac{dr}{\cij(r)},
\quad\cb_j:=\frac1{\phiij(1)},
&\quad&\phib_j(s):=\cb_j\cdot\phiij(s),
\quad 
s\in[0,1].
\end{alignat}
Then all $\phib_j:[0,1]\to[0,1]$ are Lipschitz continuous, surjective and strictly monotone, hence invertible with Lipschitz continuous inverses $(\phib_j)^{-1}:[0,1]\to[0,1]$.

Next we put 
\begin{equation}\label{eq:def-varphi-vectors}
\begin{aligned}
&\varphi^e:=(\varphi^e_1,\ldots,\varphi^e_{\ell})^\top:\RRp\to\CC^{\ell},&\quad &(\varphi^e)^{-1}:=\left((\varphi^e_1)^{-1}\ldots,(\varphi^e_{\ell})^{-1}\right)^\top:\RRp\to\CC^{\ell},\\
&\phib:=(\phib_1,\ldots,\phib_m)^\top:[0,1]\to\CC^m, &\quad &(\phib)^{-1}:=\left((\phib_1)^{-1},\ldots,(\phib_m)^{-1}\right)^\top:[0,1]\to\CC^m,\\
&\cb:=\bigl(\cb_1,\ldots,\cb_m\bigr)\in\CC^m, &\quad &\cbi:=\bigl(\cb_1^{-1},\ldots,\cb_m^{-1}\bigr)\in\CC^m,\\
&\Cb:=\diag\bigl(\cb_1,\ldots,\cb_m\bigr)\in\rM_m(\CC).&&
\end{aligned}
\end{equation}
Finally, for functions $h=(h_1,\ldots,h_n)^\top:I\subseteq\RR\to\CC^n$, $\nu=(\nu_1,\ldots,\nu_n):I\to I^n$, a vector $x=(x_1,\ldots,x_n)^\top\in\RR^n$ and scalars $r,s\in\RR$ we set
\begin{align*}
(h\circ\nu)(r)&:=\bigl(h_1(\nu_1(r)),\ldots,h_n(\nu_n(r))\bigr)^\top,\\
h(r+x)&:=\bigl(h_1(r+x_1),\ldots,h_n(r+x_n)\bigr)^\top,\\
h\bigl(r+\tfrac{s}{x}\bigr)&:=\bigl(h_1\bigl(r+\tfrac{s}{x_1}\bigr),\ldots,h_n\bigl(r+\tfrac{s}{x_n}\bigr)\bigr)^\top.
\end{align*}

Using this notation we define the transformations
\begin{equation}\label{eq:def-Q_phi}
\begin{aligned}
&\Jp\in\sL(\Xe),&\quad&\Jp\fe:=\fe\circ\varphi^e&\quad&\text{for }\fe\in\Xe,\\
&\Jpb\in\sL(\Xin), &&\Jpb\fin:=\fin\circ\phib&&\text{for }\fin\in\Xin.
\end{aligned}
\end{equation}
Then $\Jp$ and $\Jpb$ are invertible with bounded inverses $\Jpi=J_{(\varphi^e)^{-1}}\in\sL(\Xe)$ and $\Jpbi=J_{(\phib)^{-1}}\in\sL(\Xin)$. These maps will be used in \autoref{lem:rearrange-sG} to transform space dependent diffusion coefficients into constant ones.

\smallbreak
Next for fixed $\tn>0$ and $t\in[0,\tn]$ we introduce the bounded linear operators \\ $\Qt\in\sL(\rL^p([0,\tn],\CC^{\ell}),\Xe)$ and $\Rt,\St\in\sL(\rL^p([0,\tn],\CC^m),\Xin)$ by
\begin{equation}\label{eq:def-Rt-Sr}
\Qt\ue:=\ueh(t-\p\bigr),\quad
\Rt\ui:=\uih\bigl(t-\tfrac\ps\cb\bigr),
\quad\text{and}\quad
\St\ui:=\uih\bigl(t+\tfrac{\ps-1}\cb\bigr),
\end{equation}
where $\uh$ denotes the extension of a function $u$ defined on $I\subset\RR$ to $\RR$ by the value $0$. 
Observe that $\St=\psi\Rt$ for $\psi\in\sL(\Xin)$, $(\psi\fin)(\p):=\fin(1-\p)$.
Moreover, we define
\begin{equation*}
\Phieb=\bigl(\Phiebe,\Phiebi\bigr):=\Phie\cdot c(\p,\p)^{-1}
\in\sL\bigl(\CnRpCl\times\CneCm,Y_1\bigr),
\end{equation*}
where
\begin{equation}\label{eq:c(s)}
\begin{aligned}
c(\p,\p):=\sqrt{a(\p,\p)}
&=
\begin{psmallmatrix}
\qe(\ps)&0\\0&\qi(\ps)
\end{psmallmatrix}
\cdot
\begin{psmallmatrix}
\ce(\ps)&0\\0&\ci(\ps)
\end{psmallmatrix}
\cdot
\begin{psmallmatrix}
\qe(\ps)&0\\0&\qi(\ps)
\end{psmallmatrix}^{-1}\\
&=:q(\p,\p)\cdot\mu(\p,\p)\cdot q(\p,\p)^{-1},
\end{aligned}
\end{equation}
for $\ce(\ps)$ and $\ci(\ps)$ given in \eqref{eq:def-ce} and \eqref{eq:def-ci}, respectively.

\smallbreak
Now we are ready to introduce the operator $\sRtn$ as follows. Note that 
\begin{align*}
\rL^p\bigl([0,\tn],Y\bigr)&=\rL^p\bigl([0,\tn],Y_0\bigr)\times\rL^p\bigl([0,\tn],Y_1\bigr)\\
&=\rL^p\bigl([0,\tn],\CC^{\ell}\bigr)\times\rL^p\bigl([0,\tn],\CC^m\bigr)\times\rL^p\bigl([0,\tn],\CC^m\bigr).
\end{align*}

\begin{lem}\label{lem:bdd-ext}
The operator $\sRtn:\rW^{1,p}_0([0,\tn],Y)\subset\rL^p([0,\tn],Y)\to\rL^p([0,\tn],Y)$ given by
\begin{equation}\label{eq:def-sR}
\begin{aligned}
(\sRtn\su)(t):&=
\begin{psmallmatrix}
\Phieb&\Phin
\end{psmallmatrix}
\cdot
\begin{psmallmatrix}
q(\p,\p)&0\\
0&q(\p,\p)
\end{psmallmatrix}
\cdot
\begin{psmallmatrix}
\Jp&0&0\\
0&\Jpb&\Jpb\\
-\Jp&0&0\\
0&\Jpb&-\Jpb
\end{psmallmatrix}
\cdot
\begin{psmallmatrix}
\Qt&0&0\\
0&\St&0\\
0&0&\Rt\\
\end{psmallmatrix}\su\\
&=
\Bigl((\Phiebe-\Phine)\qe(\p)\Jp\cdot\Qt,(\Phiebi+\Phini)\qi(\p)\Jpb\psi\cdot\Rt,(\Phiebi-\Phini)\qi(\p)\Jpb\cdot\Rt\Bigr)\,\su
\end{aligned}
\end{equation}
is well-defined and has a unique bounded extension denoted again by $\sRtn\in\sL(\rL^p([0,\tn],Y))$.
\end{lem}

The operator $\sRtn$ plays a crucial role in our main result, see \autoref{thm:gen-Lp}.
As we will see later, in many important cases of boundary conditions involving just the boundary values, cf. \autoref{ssec:BM}, the operator $\sRtn$ reduces to a  matrix. Before starting the proof, we note that in this section we equip all subspaces $Z\subseteq\CC^n$ with  the maximum norm, i.e., we define 
\[
\bigl\|(v_1,\ldots,v_n)^\top\bigr\|_{Z}
:=\max\bigl\{|v_1|,\ldots,|v_n|\bigr\}.
\]

\begin{proof}[Proof of \autoref{lem:bdd-ext}]
Observe that 
\begin{align*}
&(\Phiebe-\Phine)\qe(\p)\Jp\in\sL\bigl(\CnRpCl,Y\bigr)
,\quad\text{and}\\
&(\Phiebi+\Phini)\qi(\p)\Jpb\psi,\ (\Phiebi-\Phini)\qi(\p)\Jpb\in\sL\l(\CneCm,Y\r)
\end{align*}
are well-defined. Hence, it suffices to show that the operator
\begin{align*}
&U_{\tn}:\rW^{1,p}_0\bigl([0,\tn],\CC^k\bigr)\subset\rL^p\bigl([0,\tn],\CC^k\bigr)\to\rL^p\bigl([0,\tn],Y\bigr),\\
&(U_{\tn} u)(t):=\Phi\,\uh\bigl(t-\theta(\p)\bigr),\ t\in[0,\tn],
\end{align*}
is well-defined and has a bounded extension in $\sL(\rL^{p}([0,\tn],\CC^k),\rL^{p}([0,\tn],Y))$ where for the
\begin{itemize}
\item external part $k=\ell$, $\Phi\in\sL(\rC_0(I,\CC^k),Y)$ and $\theta(s):=s$, $s\in I:=\RRp$, 
\item internal part $k=m$, $\Phi\in\sL(\rC(I,\CC^k),Y)$ and $\theta(s):=\tfrac s{\,\cb\,}$, $s\in I:=[0,1]$.
\end{itemize}

In both cases the assumption $u(0)=0$ implies that the function $\uh(t-\theta(\p))\in\rW^{1,p}(I,\CC^k)$ has compact support and hence $\Phi\,\uh(t-\theta(\p))$ is well-defined for all $t\in[0,\tn]$. Moreover, since $\uh|_{(-\infty,\tn]}$ is uniformly continuous, the map%
\footnote{Note that by definition $\rC_0([0,1],\CC^k)=\rC([0,1],\CC^k)$.}
$[0,\tn]\ni t\mapsto\uh(t-\theta(\p))\in\rC_0(I,\CC^k)$ is continuous and therefore $[0,\tn]\ni t\mapsto\Phi\,\uh(t-\theta(\p))\in Y$ is continuous as well. Summing up, this shows that the operator $U_{\tn}$ is well-defined.

\smallbreak
Next we verify that $U_{\tn}$ is bounded. Since $\Phi\in\sL(\rC_0(I,\CC^k),Y)$ and $Y$ is finite dimensional, by the Riesz--Markov representation theorem there exists a function $\eta:I\to\sL(\CC^k,Y)$ of bounded variation such that $\Phi$ is given by the Riemann--Stieltjes integral
\begin{equation}\label{eq:Phi-generic-e}
\Phi h=\int_I d\eta(s)\, h(s)
\quad\text{for all}\quad
h\in\rW^{1,p}(I,\CC^k).
\end{equation}
Then by Hölder's inequality and Fubini's theorem we conclude for $u\in\rW_0^{1,p}([0,\tn],\CC^k)$ that
\begin{align*}
\|U_{\tn} u\|_p^p&=
\int_{0}^{\tn}\bigl\|\Phi\, \uh\bigl(t-\theta(\p)\bigr)\bigr\|_{Y}^p\dt\\\notag
&=\int_0^{\tn}\Bigl\|\int_Id\eta(s)\,\uh\bigl(t-\theta(s)\bigr)\Bigr\|_{Y}^p\dt\\\notag
&\le\int_0^{\tn}\Bigl(\int_I\bigl\| \uh\bigl(t-\theta(s)\bigr)\bigr\|_{\CC^k}\,d|\eta|(s)\Bigr)^p\dt\\\notag
&\le\bigl(|\eta|(I)\bigr)^{p-1}\cdot\int_0^{\tn}\int_I\bigl\| \uh\bigl(t-\theta(s)\bigr)\bigr\|_{\CC^k}^p\,d|\eta|(s)\dt\\\notag
&\le\|\eta\|^{p-1}\cdot\int_I\int_0^{\tn}\bigl\| \uh\bigl(t-\theta(s)\bigr)\bigr\|_{\CC^k}^p\dt\,d|\eta|(s)\\
&\le\|\eta\|^{p}\cdot\|u\|^p_{p},
\end{align*}
where $|\eta|:I\to\RRp$ denotes the positive Borel measure defined by the total variation of 
$\eta$ and $\|\eta\|:=|\eta|(I)$. Since $\rW^{1,p}_0([0,\tn],\CC^k)$ is dense in $\rL^{p}([0,\tn],\CC^k)$ this implies that $U_{\tn}$ has a unique bounded extension as claimed.
\end{proof}

\subsection{The main result}

We are now ready to state our main generation result. 

\begin{thm}\label{thm:gen-Lp}
Let \autoref{asu:s-asu-Lp}  be satisfied.
If there exists $\tn>0$ such that the operator $\sRtn\in\sL(\rL^p([0,\tn],\CC^{\ell}\times\CC^m\times\CC^m))$ given by \eqref{eq:def-sR} is invertible,
then the operator $G$ defined in \eqref{eq:def-G-general} generates a cosine family on $X=\LpRpCl\times\LpneCm$ with phase space $V\times X$ for $V:=\ker(\Phin)$.
\end{thm}

The proof is split into four parts where in the first three we assume $B=0$. We start by showing the result under the hypothesis that the operator matrix $\sG$ in \eqref{eq:def-sG} below generates a semigroup. Then, using a series of lemmas we give the proof that $\sG$ indeed is a generator, first in case $q(\p,\p)\equiv\IId$, then for general $q(\p,\p)$. Finally, we prove the result for $B\ne0$.

\begin{proof}[Proof of \autoref{thm:gen-Lp}, 1$^{\text{st}}$  part]  Assume that $B=0$ and $q(\p,\p)\equiv\IId$. Hence, $a(\p,\p)=\lambda(\p,\p)$ and $c(\p,\p)=\mu(\p,\p)$ are diagonal matrices. By $\lambda'$ we denote the derivative of the corresponding diagonal entries, i.e.,
\[\lambda'(\p,\p):= \diag\l((\ae)'(\p),(\ai)'(\p)\r).\]

On $\sX:=X\times X$ we consider the operator matrix
\begin{equation}\label{eq:def-sG}
\sG:=\begin{pmatrix}
0&D_\Phin\\D_\Phieb&0
\end{pmatrix},
\quad
D(\sG):=D(D_\Phieb)\times D(D_\Phin),
\end{equation}
where
\begin{equation*}
\begin{aligned}
&D_\Phin:=c(\p,\p)\cdot\tdds,\  &D(D_\Phin):=\bigl\{g\in\WepRpCl\times\WepneCm:\Phin\, g=0\bigr\}=\ker(\Phin),\\
&D_\Phieb:=c(\p,\p)\cdot\tdds,\  &D(D_\Phieb):=\bigl\{f\in\WepRpCl\times\WepneCm:\Phieb f=0\bigr\}=\ker(\Phieb).
\end{aligned}
\end{equation*}

Then $\sG$ and $-\sG$ are similar via the similarity transformation induced by $\diag(\Id,-\Id)$. 
Hence, if for the time being we assume that $\sG$ generates a $C_0$-semigroup, by  \cite[Sect.~II.3.11]{EN:00}  it already generates a group. By \cite[Expl.~3.14.15]{ABHN:11} this implies that $\sG^2$ generates a cosine family with phase space $\sV\times\sX$ for $\sV:=[D(\sG)]$. However,  $\sG^2$ is given by the diagonal matrix with diagonal domain
\begin{equation}\label{eq:def-sG2}
\sG^2:=\begin{pmatrix}
D_\Phin D_\Phieb&0\\0&D_\Phieb D_\Phin
\end{pmatrix},
\quad
D(\sG^2):=D(D_\Phin D_\Phieb)\times D(D_\Phieb D_\Phin).
\end{equation}
Hence, $\Gt:=D_\Phieb D_\Phin$  generates a cosine family with phase space $V\times X$ for $V:=[D(D_\Phin)]=\ker(\Phin)$.
Since $\lambda$ is Lipschitz continuous, $\lambda'\in\rL^\infty(\RRp\times [0,1],\rM_{\ell+m}(\CC))$ induces a bounded multiplication operator on $X$ and therefore $P:=\frac{\lambda'}2\cdot\dds\in\sL(V,X)$. However, by \autoref{cor:reg-prod}, $D(G)=D(\Gt)$ and $G=\Gt-P$, hence by \cite[Cor.~3.14.13]{ABHN:11} it follows that $G$ generates a cosine family with phase space $V\times X$ as claimed. 
\end{proof}

Next we verify the generator property of the matrix $\sG$.  To do so we proceed in several steps. First we assume again that 
$c(\p,\p)=\diag(\ce(\p),\ci(\p))$ is diagonal, i.e.
 $q(\p,\p)\equiv\IId$. The case of general $q(\p,\p)$ as in \eqref{eq:rep-a(s)} then follows by similarity and bounded perturbation.

\smallbreak
We start by simplifying $\sG$ by rearranging the coordinates of $\sX$ and by normalizing the matrices $\ce(\p),\ci(\p)$. Recall that $\Cb$ is defined in \eqref{eq:def-varphi-vectors} and
$\Jp\in\sL(\Xe)$, $\Jpb\in\sL(\Xin)$ are given by \eqref{eq:def-Q_phi}.

\begin{lem}\label{lem:rearrange-sG}
Let $q(\p,\p)=\IId$. Then the operator matrix $\sG$ on $\sX=X\times X=(\Xe\times\Xin)\times(\Xe\times\Xin)$ given in \eqref{eq:def-sG} is similar to $\sGr$ on $\sXr:=\sXe\times\sXi=(\Xe\times\Xe)\times(\Xin\times\Xin)$ where
\begin{equation}\label{eq:def-sGr}
\sGr:=
\diag\Bigl(\tdds,-\tdds,\Cb\cdot\tdds,-\Cb\cdot\tdds\Bigr),
\quad
D(\sGr):=\ker(\Phix)
\end{equation}
and
\begin{align}\label{eq:def-PHIx}
\Phix&:=\bigl((\Phiebe+\Phine)\cdot \Jp,(\Phiebe-\Phine)\cdot \Jp,(\Phiebi+\Phini)\cdot \Jpb,(\Phiebi-\Phini)\cdot \Jpb\bigr)\\\notag
&\in\sL\bigl(\CnRpCl\times\CnRpCl\times\CneCm\times\CneCm,Y\bigr).
\end{align}

\end{lem}

\begin{proof} Consider the invertible transformation
\begin{align*}
\sS&:=\tfrac1{2}\cdot
\begin{psmallmatrix}
\Jp&\Jp&0&0\\
0&0&\Jpb&\Jpb\\
\Jp&-\Jp&0&0\\
0&0&\Jpb&-\Jpb
\end{psmallmatrix}
\in\sL\bigl(\sXr,\sX\bigr)
\qquad\text{with inverse}\\
\sS^{-1}&:= 
\begin{psmallmatrix}
\Jpi&0&\Jpi&0\\
\Jpi&0&-\Jpi&0\\
0&\Jpbi&0&\Jpbi\\
0&\Jpbi&0&-\Jpbi
\end{psmallmatrix}
\in\sL\bigl(\sX,\sXr\bigr).
\end{align*}
We claim that
\begin{equation}\label{eq:sim-sG-sGr}
\sGr=\sS^{-1}\cdot\sG\cdot\sS.
\end{equation}
Since $\rg(\Phine,\Phini)\subseteq Y_0$, $\rg(\Phiebe,\Phiebi)\subseteq Y_1$  and $Y_0\cap Y_1=\{0\}$ we have
$D(\sG)=\ker(\Phiebe,\Phiebi,\Phine,\Phini)$. Using this a simple computation shows that $D(\sS^{-1}\sG\sS)=\ker((\Phiebe,\Phiebi,\Phine,\Phini)\cdot\sS)=D(\sGr)$.
Next $(\phi^{-1})'=\ce(\p)\circ\phi^{-1}$ implies 
\begin{equation*}
\Jp\cdot\tdds\cdot \Jp^{-1}=
\ce(\p)\cdot\tdds.
\end{equation*}
Similarly, since $(\Cb\cdot(\phib)^{-1})'=\ci(\p)\circ(\phib)^{-1}$ we obtain
\begin{equation*}
\Jpb\cdot\Cb\cdot\tdds\cdot \Jpb^{-1}=
\ci(\p)\cdot\tdds.
\end{equation*}
This implies that $\sGr\binom{\fin}{\fe}=\sS^{-1}\sG\sS\binom{\fin}{\fe}$ for $\binom{\fin}{\fe}\in D(\sGr)$ which completes the proof of \eqref{eq:sim-sG-sGr}.
\end{proof}

We now represent $\sGr$ as a domain perturbation of a simpler generator $\sA$ which can be treated by a (slight modification of a) recent perturbation result from \cite{ABE:13} (see \autoref{thm:pert-bc}). 

Thanks to \autoref{lem:rearrange-sG} we can consider the external and internal part separately.

\subsubsection*{External Part.} We introduce on $\Xe=\LpRpCl$ the operators
\begin{alignat}{3}
\Dme&:=\tdds,&\qquad&D(\Dme):=\WepRpCl,\\
\Dne&:=\tdds,&&D(\Dne):=\bigl\{f\in D(\Dme):f(0)=0\bigr\}\label{eq:def-Dne},
\end{alignat}
and define on $\sXe=\Xe\times\Xe$ the operator matrices
\begin{alignat}{3}
\sAme&:=
\diag\bigl(\Dme,-\Dme\bigr)
\label{eq:def-sAme}
,&\qquad&D(\sAme):=D(\Dme)\times D(\Dme),\\
\sAe&:=
\diag\bigl(\Dme,-\Dne\bigr)
,&&D(\sAe):=D(\Dme)\times D(\Dne).\label{eq:def-sAe}
\end{alignat}

Note that $\Dme$ and $-\Dne$ generate the strongly continuous left- and right-shift semigroups $\Tlet$ and $\Tret$ on $\Xe$, respectively, given by
\begin{equation}\label{eq:def-Telr}
\bigl(\Tle(t)f\bigr)(\p):=f(\p+t),
\qquad
\bigl(\Tre(t)f\bigr)(\p):=
\fh(\p-t),
\end{equation}
where $\fh$ denotes the extension of the function $f:\RRp\to\CC^{\ell}$ to $\RR$ by the value $0$. 
This gives immediately the following result.

\begin{lem}\label{lem:sAe-gen}
The operator $\sAe$ defined in \eqref{eq:def-sAe} generates a $C_0$-semigroup $\sTet$ given by
\begin{equation}\label{eq:def-sTt-e}
\sTe(t)=
\diag\bigl(\Tle(t),\Tre(t)\bigr),\quad t\ge0.
\end{equation}
\end{lem}

Note that in the context of \autoref{sec:dom-pert} we have $\sAe\subset\sAme$ with domain
\begin{align*}
D(\sAe)&=\Bigl\{\tbinom{f}{g}\in D(\sAme):\sLe\tbinom{f}{g}=0\Bigr\}=\ker(\sLe)
\end{align*}
for
\begin{equation}\label{eq:def-sL-e}
\sLe:=(0,\delta_0)\in\sL\bigl([D(\sA_m)],\partial\sXe\bigr),
\end{equation}
where $\delta_0$ denotes the point evaluation in $s=0$ and  $\partial\sXe:=\CC^{\ell}$. 
Now the following follows easily by inspection.

\begin{lem}\label{lem:BCS-wp-e}
Let the operators $\sAme$ and $\sLe$ be defined by \eqref{eq:def-sAme} and \eqref{eq:def-sL-e}, respectively. 
Then for $\tn>0$ and given $u\in\rW^{2,p}_0([0,\tn],\partial\sXe)$ the function $\sx:[0,\tn]\to\sXe=\Xe\times\Xe$ defined by
\begin{equation}\label{eq:sol-BCS-e}
\sx(t,\p):=
\bigl(0,\uh(t-\p)\bigr)^\top
\end{equation}
is a classical solution of the boundary control system
\begin{equation}
\label{eq:BCSe}
\begin{cases}
\dot\sx(t)=\sAme\sx(t),&0\le t\le\tn,\\
\sLe\sx(t)=u(t),&0\le t\le\tn,\\
\sx(0)=0.
\end{cases}
\end{equation}
\end{lem}

\subsubsection*{Internal Part.} 
Recall that $\Cb:=\diag(\cb_1,\ldots,\cb_m)$ with $\cb_i$ defined in \eqref{eq:def-varphi-j}.
 Then we introduce on $\Xin=\LpneCm$ the operators
\begin{alignat}{3}
\Dmi&:=\Cb\cdot\tdds,&\qquad&D(\Dmi):=\WepneCm,\\
\Dni&:=\Cb\cdot\tdds,&&D(\Dni):=\bigl\{f\in D(\Dmi):f(0)=0\bigr\}\label{eq:def-D_0},\\
\Dei&:=\Cb\cdot\tdds,&&D(\Dei):=\bigl\{f\in D(\Dmi):f(1)=0\bigr\}\label{eq:def-D_1},
\end{alignat}
and define on $\sXi=\Xin\times\Xin$ the operator matrices
\begin{alignat}{3}
\sAmi&:=
\begin{pmatrix}
\Dmi&0\\0&-\Dmi
\end{pmatrix}\label{eq:def-sAmi}
,&\qquad&D(\sAmi):=D(\Dmi)\times D(\Dmi),\\
\sAi&:=
\begin{pmatrix}
\Dei&0\\0&-\Dni
\end{pmatrix}
,&&D(\sAi):=D(\Dei)\times D(\Dni).\label{eq:def-sAi}
\end{alignat}
Then $\Dei$ and $-\Dni$ generate the strongly continuous nilpotent left- and right-shift semigroups $\Tlit$ and $\Trit$ on $\Xin$, respectively, given by
\begin{equation}\label{eq:def-Tilr}
\bigl(\Tli(t)f\bigr)(\p):=\fh(\p+\cb\cdot t),
\qquad
\bigl(\Tri(t)f\bigr)(\p):=
\fh(\p-\cb\cdot t),
\end{equation}
where $\cb=(\cb_1,\ldots,\cb_m)$.
This gives immediately the following result.

\begin{lem}\label{lem:sAi-gen}
The operator $\sAi$ defined in \eqref{eq:def-sAi} generates a $C_0$-semigroup $\sTit$ given by
\begin{equation}\label{eq:def-sTti}
\sTi(t)=
\begin{pmatrix}
\Tli(t)&0\\0&\Tri(t)
\end{pmatrix},
\quad t\ge0.
\end{equation}
\end{lem}

As before we observe that in the context of \autoref{sec:dom-pert} we have $\sAi\subset\sAmi$ with domain
\begin{align*}
D(\sAi)&=\Bigl\{\tbinom{f}{g}\in D(\sAmi):\sLi\tbinom{f}{g}=0\Bigr\}=\ker(\sLi),
\end{align*}
for
\begin{equation}\label{eq:def-sL-sC}
\sLi:=
\begin{pmatrix}
\delta_1&0\\
0&\delta_0
\end{pmatrix}
\in\sL\bigl([D(\sAmi)],\partial\sXi\bigr),
\end{equation}
where $\delta_s$ denotes the point evaluation in $s\in\{0,1\}$ and  $\partial\sXi:=\CC^m\times\CC^m$.

\begin{lem}\label{lem:BCS-wp}
Let the operators $\sAmi$ and $\sLi$ be defined by \eqref{eq:def-sAmi} and \eqref{eq:def-sL-sC}, respectively. 
Then for $\tn:=\min\{\cbi_1,\ldots,\cbi_m\}>0$ and given $\sv=(v,w)^\top\in\rW^{2,p}_0([0,\tn],\partial\sXi)$ the function $\sx:[0,\tn]\to\sXi=\Xin\times\Xin$ defined by%
\begin{equation}\label{eq:sol-BCS}
\sx(t,s):=
\Bigl(\vh\bigl(t+\tfrac{s-1}\cb\bigr),\wh\bigl(t-\tfrac s{\,\cb\,}\bigr)\Bigr)^\top,
\quad t\in[0,\tn],\ s\in[0,1] 
\end{equation}
is a classical solution of the boundary control system
\begin{equation}
\label{eq:BCS}
\begin{cases}
\dot\sx(t)=\sAmi\sx(t),&0\le t\le\tn,\\
\sLi\sx(t)=\sv(t),&0\le t\le\tn,\\
\sx(0)=0.
\end{cases}
\end{equation}
\end{lem}

We are now well-prepared to continue the proof of our main result. 

\begin{proof}[Proof of \autoref{thm:gen-Lp}, 2$^{\text{nd}}$ part]
We show that $\sGr$ given by \eqref{eq:def-sGr} generates a semigroup on $\sXr:=\sXe\times\sXi$. By \autoref{lem:rearrange-sG} and the first part  of the proof this proves \autoref{thm:gen-Lp} in case $q(\p,\p)=\IId$ and $B=0$.

\smallbreak
For the operators $\sAme$, $\sAe $ and $\sAmi$, $\sAi$ given by \eqref{eq:def-sAme}--\eqref{eq:def-sAe} and \eqref{eq:def-sAmi}--\eqref{eq:def-sAi}, respectively, we define on $\sXr$ the matrices
\begin{alignat}{3}\notag
&\sA_m:=\diag\bigl(\sAme,\sAmi\bigr),&\quad&
D(\sA_m):=D(\sAme)\times D(\sAmi),\\ \label{eq:def-sA}
&\sA:=\diag\bigl(\sAe,\sAi\bigr),&&D(\sA):=D(\sAe)\times D(\sAi).
\end{alignat}
Then $\sGr,\sA\subset\sA_m$ with domains $D(\sGr)=\ker(\Phix)$ and $D(\sA)=\ker(\sL)$, where
$\Phix$ is given by \eqref{eq:def-PHIx}
in \autoref{lem:rearrange-sG} and
\begin{equation*}
\sL:=\diag\bigl(\sLe,\sLi\bigr)=
\begin{pmatrix}
0&\delta_0&0&0\\
0&0&\delta_1&0\\
0&0&0&\delta_0
\end{pmatrix}
:D(\sA_m)\to\partial\sX:=\partial\sXe\times\partial\sXi=\CC^{\ell+2m}.
\end{equation*}
Moreover, by \autoref{lem:sAe-gen} and \autoref{lem:sAi-gen}, $\sA$ generates a $C_0$-semigroup $\sTt$ given by
\begin{equation}\label{eq:def-sTt}
\sT(t)=\diag\bigl(\sTe(t),\sTi(t)\bigr),\quad t\ge0.
\end{equation}
Hence, the assertion follows if we verify the assumptions (i)--(iv) in \autoref{thm:pert-bc-v2} adapted to the present context. Let $\tn:=\min\{\cbi_1,\ldots,\cbi_m\}$.

\smallbreak
(i) For $t\in[0,\tn]$ and $u\in\rL^{p}([0,\tn],\partial\sXe)$ define $\sBte u\in\sXe=\Xe\times\Xe$ by the right-hand-side of \eqref{eq:sol-BCS-e}. Similarly, for
$\sv=(v,w)^\top\in\rL^p([0,\tn],\partial\sXi)$ define $\sBti\sv\in\sXi=\Xin\times\Xin$ by the right-hand-side of \eqref{eq:sol-BCS} and put 
\[\sBt:=\diag\bigl(\sBte,\sBti\bigr):\partial\sX=\partial\sXe\times\partial\sXi\to\sXr=\sXe\times\sXi.\]
Then $\Bttn\subset\sL(\rL^p([0,\tn],\partial\sX),\sXr)$ is strongly continuous. Moreover, by \autoref{lem:BCS-wp-e} and \autoref{lem:BCS-wp}, $\sx(t):=\sBt(u,\sv)^\top$ solves for given $(u,\sv)\in\rW^{2,p}_0([0,\tn],\partial\sXe)\times\rW^{2,p}_0([0,\tn],\partial\sXi)$
\begin{equation}
\label{eq:BCSn}
\begin{cases}
\dot\sx(t)=\sA_m\sx(t),&0\le t\le\tn,\\
\sL\sx(t)=\tbinom{u(t)}{\ssv(t)}
,&0\le t\le\tn,\\
\sx(0)=0,
\end{cases}
\end{equation}
hence the assertion follows from \autoref{lem:admiss-LA}.

\smallbreak
(ii) Using the terminology introduced in \autoref{rem:C-admiss},
we have to show that $\Phix$ in \eqref{eq:def-PHIx} is $p$-admissible for the semigroup $\sTt$ generated by $\sA$. By the representations of $\sT(t)$ in \eqref{eq:def-sTt} this follows if we verify the $p$-admissibility of every
\begin{itemize}
\item  $\Phi\in\sL(\rC_0(I,\CC^k),Y)$ for $I=\RRp$ and $k= \ell$ with respect to the semigroups $\Tlet$ and $\Tret$ given in \eqref{eq:def-Telr}, and
\item $\Phi\in\sL(\rC(I,\CC^k),Y)$ for $I=[0,1]$ and $k=m$ with respect to the semigroups $\Tlit$ and $\Trit$ given in \eqref{eq:def-Tilr}.  
\end{itemize}
Let $T(t)\in\{\Tle(t),\Tre(t),\Tli,\Tri\}$ and define
\[
c:=\begin{cases}
(1,\ldots,1)\in\CC^{\ell}&\text{in the external case},\\
\cb=(\cb_1,\ldots,\cb_m)\in\CC^m&\text{in the internal case}.
\end{cases}
\]
Then similarly as in the proof of \autoref{lem:bdd-ext} it follows that
\begin{align*}
\int_0^{\tn}\bigl\|\Phi\,T(t)f\bigr\|_{Y}^p\dt
&=\int_0^{\tn}\Bigl\|\int_Id\eta(s)\,\fh(s\pm c\cdot t)\Bigr\|_{Y}^p\dt\\
&\le\|\eta\|^{p-1}\cdot\int_I\int_0^{\tn}\bigl\| \fh(s\pm c\cdot t)\bigr\|_{\CC^k}^p\dt\,d|\eta|(s)\\
&\le{\|\eta\|^p}\cdot M^p\cdot\|f\|^p_{p},
\end{align*}
where $\eta$ is given by \eqref{eq:Phi-generic-e} and $M:=\|c^{-1}\|_{\CC^k}$.
This completes the proof of (ii).

\smallbreak
(iii)\&(iv)  By (i), \autoref{lem:BCS-wp-e}, \autoref{lem:BCS-wp} and \autoref{lem:admiss-LA} the controllability maps for the problem \eqref{eq:BCSn} are given by
\begin{equation*}
\sBt=
\begin{pmatrix}
0&0&0\\
\Qt&0&0\\
0&\St&0\\
0&0&\Rt
\end{pmatrix}
\in\sL\bigl(\rL^p\bigl([0,\tn],\partial\sX\bigr),\sXr\bigr),
\quad t\in[0,\tn],
\end{equation*}
where $\Qt,\St$ and $\Rt$ are defined in \eqref{eq:def-Rt-Sr}. Hence, for all $\su\in\rW_0^{2,p}([0,t_0],Y)$,
\begin{equation*} 
\sQ_{\tn}\su =\Phi\sBt \su = \sRtn \su,
\end{equation*}
where the last equality is obtained by direct computation using definitions of  $\Phix$ in \eqref{eq:def-PHIx} and  $\sRtn$ in \eqref{eq:def-sR}.
This combined with \autoref{lem:bdd-ext} implies (iii)  and also (iv) follows immediately from the invertibility assumption on $\sR_{\tn}$.

\smallbreak
Summing up we conclude that for $q(\p,\p)\equiv \IId$ the matrix $\sGr$ given in \eqref{eq:def-sGr}, hence by the similarity in \autoref{lem:rearrange-sG} also $\sG$ in \eqref{eq:def-sG}, generate $C_0$-semigroups if $\sRtn$ given by \eqref{eq:def-sR} (for $q(\p,\p)\equiv\IId$) is invertible.
\end{proof}

Next we consider general $q(\p,\p)$, i.e., the case of possibly non-diagonal diffusion coefficient matrices $a(\p,\p)$.

\begin{proof}[Proof of \autoref{thm:gen-Lp}, 3$^{\text{rd}}$ part]
Assume that $a(\p,\p), c(\p,\p)$ are given by \eqref{eq:rep-a(s)} and \eqref{eq:c(s)}, respectively, where $q(\p,\p),q(\p,\p)^{-1}$ are Lipschitz continuous and bounded.
Then via the similarity transformation induced by $\diag(q(\p,\p),q(\p,\p))$ we obtain that the operator matrix 
\begin{equation*}
\sG=
\begin{psmallmatrix}
0&c(\p,\p)\cdot\dds\\
c(\p,\p)\cdot\dds&0
\end{psmallmatrix},
\quad D(\sG)=\ker\bigl(\Phie\cdot c(\p,\p)^{-1}\bigr)\times\ker(\Phin),
\end{equation*}
defined in \eqref{eq:def-sG}, is similar to the operator matrix
\begin{equation*}
\begin{aligned}
\sG&\simeq\begin{psmallmatrix}
0&\mu(\ps,\ps)q^{-1}(\ps,\ps) \dds q(\ps,\ps)\\
\mu(\ps,\ps)q^{-1}(\ps,\ps) \dds q(\ps,\ps)&0
\end{psmallmatrix} \\
&=
\begin{psmallmatrix}
0&\mu(\p,\p)\cdot\dds\\
\mu(\p,\p)\cdot\dds&0
\end{psmallmatrix}+\begin{psmallmatrix}
0&\mu(\ps,\ps)\u^{-1}(\ps,\ps)\u'(\ps,\ps)\\
\mu(\ps,\ps)\u^{-1}(\ps,\ps)\u'(\ps,\ps)&0
\end{psmallmatrix}\\
&=:\hat\sG +\sP,
\end{aligned}
\end{equation*}
where we used that $q(\p,\p)^{-1}c(\p,\p)=\mu(\p,\p)q(\p,\p)^{-1}$ and 
$\dds q(\p,\p) = q'(\p,\p) + q(\p,\p) \dds$.
Since $q$ and $q^{-1}$ are Lipschitz continuous and bounded, we have $\sP\in\sL(X\times X)$.
Moreover, note that 
\begin{equation*}
 D(\hat\sG)=\ker\bigl(\hat\Phi_1\cdot \mu(\p)^{-1}\bigr)\times\ker(\hat\Phi_0),
\end{equation*}
for $\hat\Phi_1:=\Phie\cdot q(\p,\p)$, $\hat\Phi_0:=\Phin\cdot q(\p,\p)$.
Hence, by similarity and bounded perturbation $\sG$ is a generator iff $\hat\sG$ is. 
However, by what we proved previously for $q(\p,\p)=\IId$, i.e., for diagonal $c(\p,\p)=\mu(\p,\p)$, the operator $\hat\sG$ is a generator if $\sRtn$ given by \eqref{eq:def-sR} is invertible.
\end{proof}

We conclude the proof by considering non-zero boundary operators $B\in\sL(X)$.

\begin{proof}[Proof of \autoref{thm:gen-Lp}, 4$^{\text{th}}$ part] It remains to prove the result for $B\ne0$ satisfying the regularity condition 
\begin{equation}\label{eq:B-reg-cond}  
B\bigl(\WepRpCl\times\WepneCm\bigr)\subseteq\WepRpCl\times\WepneCm.
\end{equation}
To this end we put $\Beq:=c(\p,\p)\cdot B\in\sL(X)$ and perturb the matrix $\sG$ in \eqref{eq:def-sG} by
\[
\sB:=
\begin{pmatrix}
0&\Beq\\0&0
\end{pmatrix}
\in\sL(\sX).
\]
Then, by Part~3 and the bounded perturbation theorem, $\sG_B:=\sG+\sB$ generates a group. Now a simple computation using \autoref{cor:reg-prod} shows that
\begin{equation*} 
\sG_B^2:=\begin{pmatrix}
(D_\Phin+\Beq) D_\Phieb&0\\0&\Gt_B
\end{pmatrix},
\quad
D(\sG_B^2):=D\bigl(D_\Phin D_\Phieb\bigr)\times D(\Gt_B),
\end{equation*}
where
\begin{align*}
\Gt_B:&=D_\Phieb(D_\Phin+\Beq),\\
D\bigl(\Gt_B\bigr):&=\bigl\{f\in\WzpRpCl\times\WzpneCm:\Phin f=0,\;\Phie(f'+Bf)=0\bigr\}.
\end{align*}
Hence,  $\Gt_B$ with domain $D(\Gt_B)=D(G)$ generates a cosine family on $X$ with phase space $V\times X$ for $V:=[D(D_\Phin)]=\ker(\Phin)$.
Now as in Part~1 we have $P:=\frac{a'}2\cdot\dds\in\sL(V,X)$. Moreover, since $B\in\sL(X)$, the regularity property \eqref{eq:B-reg-cond}  
combined with \autoref{cor:reg-prod} and the closed graph theorem  imply $\Beq\in\sL(\WepRpCl\times\WepneCm)$. Hence, $Q:=c\cdot\dds\cdot\Beq\in\sL(V,X)$ and since
$G=\Gt_B-P-Q$, the claim follows from \cite[Cor.~3.14.13]{ABHN:11} as in Part~1.
\end{proof}

\begin{rem}
(i) As in the 1$^{\text{st}}$  part of the proof of \autoref{thm:gen-Lp} one can see that operators $A:=a(\p,\p)\cdot\tddss$ and $A_d:=\tdds\left( a(\p,\p)\cdot\tdds\right)$ differ only by a bounded perturbation from $V\to X$. Hence, $A$ generates a cosine family on $X$ with phase space $V\times X$ iff $A_d$ generates a cosine family on $X$ with phase space $V\times X$. 

\smallbreak
(ii)  Note that by \cite[Cor.~3.14.13]{ABHN:11} the sum $G+P$ of the generator $G$ of a cosine family with phase space $V\times X$ and a perturbation $P\in\sL(V,X)$ still generates a cosine family with the same phase space. Here in the context of \autoref{thm:gen-Lp} we have $V\inc\WepRpCl\times\WepneCm$ which implies $(\dds,\dds)\in\sL(V,X)$.
Thus, boundedness and invertibility of $\sRtn$ in \eqref{eq:def-sR} imply that for arbitrary $b(\p,\p),q(\p,\p)\in\rL^\infty(\RR_+,\CC^{\ell})\times \rL^\infty([0,1],\CC^m)$  also $G_P:=G+P$ for
\begin{equation*} 
P:=b(\p,\p)\cdot\tdds+q(\p,\p)
\end{equation*}
with domain $D(G_P):=D(G)$ generates a cosine family with the same phase space.

\smallbreak
(iii) By \cite[Thm.~3.14.17]{ABHN:11} every generator of a cosine family generates an analytic semigroup of angle $\frac\pi2$. Hence,  the previous remark gives also conditions implying that $G+P$  generates an analytic semigroup of angle $\frac\pi2$.

\smallbreak
(iv) It is quite remarkable that $\sG$ in \eqref{eq:def-sG} might generate a $C_0$-semigroup even if none of its entries $D_\Phin$ and $D_\Phieb$ are generators. For example for  $\ell=0$, $m=1$ and $Y_0=\CC^{\ell+2m}=\CC^2$, $Y_1=\{0\}$ we can choose $c(\p)\equiv1$, $\Phin=\binom{\delta_0}{\delta_1}$, $\Phieb=0$ and $B=0$. Then $\sigma(D_\Phin)=\sigma(D_\Phieb)=\CC$ and hence both operators do not generate semigroups. On the other hand, the assumptions of \autoref{thm:gen-Lp} are satisfied, hence $G=D_\Phieb\cdot D_\Phin=\Delta_D$, i.e., the Laplacian with Dirichlet boundary conditions, generates a cosine family on $\rL^p[0,1]$ for all $p\in[1,\infty)$. Similarly, by reversing the roles of $\Phin$ and $\Phieb$ (or by looking at the upper diagonal entry $D_\Phin\cdot D_\Phieb$ of $\sG^2$) it follows that also $G=\Delta_N$, i.e., the Laplacian with Neumann boundary conditions, generates a cosine family on $\rL^p[0,1]$. Similarly, choosing  $\ell=1$ and $m=0$ it follows easily that the Laplacian with Dirichlet or Neumann boundary conditions generates a cosine family on $\rL^p(\RRp)$.

\smallbreak
(v) Another remarkable fact is that while $\sA$ in \eqref{eq:def-sA} only generates a semigroup, its perturbation $\sGr$ in \eqref{eq:def-sGr} always generates a group.

\smallbreak
(vi)  We mention that even for smooth positive definite valued $a(\p,\p)$ a representation as in \eqref{eq:rep-a(s)} is not always possible.
Assume for simplicity that  $\ell=0$\footnote{In the case of empty external part ($\ell=0$) we write $q(\p)= q^i(\p)$,  $\lambda(\p)= \lambda^i(\p)$, $\mu(\p)= \mu^i(\p)$, $f=f^i$ and  $a(\p)$ and $c(\p)$ instead of $a(\p,\p)$ and $c(\p,\p)$.}. If $q(\p)\in\rC^\infty[0,1]$ such that $0<q(s)\le\frac12$ for $s\ne\frac12$ and $q^{(k)}(\frac12)=0$ for all $k=0,1,2,\ldots$, then
\[
a(s):=
\begin{cases}
\begin{psmallmatrix}
1+q(s)&0\\0&1-q(s)
\end{psmallmatrix}
&\text{if }s\in[0,\frac12],\\[6pt]
\begin{psmallmatrix}
1&q(s)\\q(s)& 1
\end{psmallmatrix}
&\text{if }s\in(\frac12,1],
\end{cases}
\]
cannot be diagonalized by means of a continuous $q(\p)$ even though $a(\p)\in\rC^\infty([0,1],\CC^2)$. This follows from the fact that the eigenvectors of $a(s)$ are given by
\begin{equation*}
\begin{cases}
\binom10,\;\binom01&\text{if }s\in[0,\frac12],\\[2pt]
\binom11,\;\binom1{-1}&\text{if }s\in(\frac12,1].
\end{cases}
\end{equation*}
However, if $a(\p)$ is analytic or if $a(s)$ has $m$ distinct eigenvalues for all $s\in[0,1]$ then it can always represented as in \eqref{eq:rep-a(s)}. 
Nevertheless, in case $p=2$ one can drop this assumption for an important class of boundary functionals, cf. \cite{Eng:13}. 
\end{rem}
\smallskip

Feller \cite{Fe:57} has characterized the boundary conditions in the domain of the generator 
of the transition semigroup corresponding to one-dimensional diffusion processes. Besides Dirichlet and Neumann boundary conditions (which we discuss in \autoref{ex-scalar-Dirichlet}), these include also non-local integral conditions which we discuss next. Note that this also generalizes the well-posednes results in \cite{MN:14}.
 
\begin{exa}\label{ex-nonlocal-interval}
We consider a diffusion operator $G\subseteq \tddss$ on $\rL^p[0,1]$ with  non-local boundary conditions. More precisely, for $\h_0,\h_1\in\rL^q[0,1]$ where $\frac1p+\frac1q=1$, we define the domain
\[D(G):=\left\{f\in\rW^{2,p}[0,1] \,\big\mid\, f(j) = \int_0^1\h_j(s)f(s) \ds,\, j=0,1 \right\}.\]
In our setting this corresponds to $\ell=0$,  $m=1$, the diffusion coefficient $a(\p)\equiv1$, the state space $X= \rL^{p}[0,1]$, the boundary spaces $Y_1 = \{0\}$, $Y_0=\CC^{2}$, and the boundary functionals $ \Phie = 0$, 
$ \Phin = \binom{\delta_1 - \sV_1}{\sV_0-\delta_0 }$, where 
$\sV_jf:=\int_0^1\h_j(s)f(s) \ds$.
This implies $\Jpb =\Id$, $\cb=1$, $q(\p)\equiv1$,
and for the operators $R_t,S_t$ defined in \eqref{eq:def-Rt-Sr} 
we obtain for $u\in\rL^{p}[0,\tn]$, $0<\tn<1$, 
\begin{equation*}
\delta_0\,\Rt u=  \delta_1\,\St u = u(t) \quad\text{ and }\quad
\delta_1\,\Rt u=\delta_0\,\St u=0,\quad t\in[0,\tn].
\end{equation*}
 Moreover, a simple computation shows that
\begin{align*}
\sV_j\Rt u&=\int_0^t\h_j(s)\,u(t-s)\ds=(\h_j*u)(t)=:(\sK_j u)(t),\\
\sV_j\St u&=\sV_j\psi\Rt u=\int_0^t(\psi \h_j)(s)\,u(t-s)\ds=\bigl((\psi\h_j)*u\bigr)(t)=:(\tilde\sK_j u)(t).
\end{align*}
Hence, the operator $\sRtn$ in \eqref{eq:def-sR} is given by
\begin{equation*}
\sRtn=\Id-
\begin{pmatrix}
\tilde\sK_1&-\sK_1\\
\tilde\sK_0&\sK_0
\end{pmatrix},
\end{equation*}
where by Young's inequality each convolution operator $\sK\in\{\sK_j,\tilde\sK_j:j=0,1\}\subset\sL(\rL^p[0,\tn])$ with kernel $h\in\{h_j,\psi\h_j:j=0,1\}\subset\rL^q[0,1]\subset\rL^1[0,1]$ satisfies
\[
\|\sK\|_{\sL(\rL^p[0,\tn])}\le\bigl\|h|_{[0,\tn]}\bigr\|_1\to0\quad\text{as }\tn\to0.
\]
This implies that $\sRtn$ is invertible for $\tn\in(0,1]$ sufficiently small and by  \autoref{thm:gen-Lp} we conclude that $G$ generates a cosine family on $X$.
\end{exa}

We close this section by considering two very common and important classes of boundary conditions. The first one uses a set of ``boundary matrices'' \eqref{eq:bound-matr} to impose the values in the end points, the second one uses two ``boundary spaces'' $Y_0,Y_1$  instead.
As we will see, in both cases our main assumption in \autoref{thm:gen-Lp}, the invertibility of the map $\sRtn$  given by \eqref{eq:def-sR}, reduces to a condition which can be easily verified. More precisely, in the first case we obtain the determinant condition \eqref{eq:cond-det}, in the second one the direct sum condition \eqref{eq:intersec-0}.

\subsection{Boundary conditions via ``boundary matrices''}
\label{ssec:BM}

For $k_0,k_1\in\NN_0$ satisfying $k_0+k_1=\ell+2m$ we choose matrices
\begin{equation}\label{eq:bound-matr}
\Vne\in\rM_{k_0 \times \ell},\quad
\Wne\in\rM_{k_1 \times \ell},\quad
\Vni,\Vei\in\rM_{k_0\times m}(\CC),\quad
\Wni,\Wei\in\rM_{k_1\times m}(\CC).
\end{equation}
Moreover, we define 
\begin{equation*}
\Wneb:=\Wne\cdot\ce(0)^{-1} ,\qquad
\Wnib:=\Wni\cdot\ci(0)^{-1},\qquad \Weib:=\Wei\cdot\ci(1)^{-1},
\end{equation*}
for $\ce(\ps)$ and $\ci(\ps)$ given in \eqref{eq:def-ce} and \eqref{eq:def-ci}, respectively.
Next we will use the matrices $\Vne,\Vni,\Vei$ to specify $k_0$ conditions containing only values at the endpoints, while the matrices $\Wne, \Wni, \Wei$ will determine $k_1$ (linear independent) conditions regarding derivatives at the endpoints.

\begin{cor}\label{cor:gen-V-W}
Let $a(\p,\p)$ be as in \eqref{eq:rep-a(s)} and assume that $\Be\in\sL(\LpRpCl,\rL^{p}(\RRp,\CC^{k_1}))$ and
$\Bi\in\sL(\LpneCm,\rL^{p}([0,1],\CC^{k_1}))$ map $\rW^{1,p}$ into $\rW^{1,p}$, i.e., \[
\Be\rW^{1,p}\bigl(\RRp,\CC^{\ell}\bigr)\subseteq\rW^{1,p}\bigl(\RRp,\CC^{k_1}\bigr)
\quad\text{and}\quad
\Bi\rW^{1,p}\bigl([0,1],\CC^m\bigr)\subseteq\rW^{1,p}\bigl([0,1],\CC^{k_1}\bigr).
\]
For $f=\binom{\fe}{\fin}\in\WzpRpCl\times\WzpneCm$ consider the boundary conditions
\begin{equation}\label{eq:BC-gen}
\begin{cases}
\Vne\fe(0)+\Vni\fin(0)+\Vei\fin(1)=0,\\
\Wne(\fe)'(0)+\Wni(\fin)'(0)-\Wei(\fin)'(1)+(\Be\fe)(0)+(\Bi\fin)(0)=0.
\end{cases}
\end{equation}

If the determinant
\begin{equation}\label{eq:cond-det}
\det
\begin{pmatrix}
\Vne&\Vei&\Vni\\[2pt]
\Wneb&\Weib&\Wnib
\end{pmatrix}
\ne0,
\end{equation}
then the operator
\begin{equation}\label{eq:def-G-delta}
\begin{aligned}[3]
G&:=a(\p,\p)\cdot\tddss,\\
D(G)&:=\biggl\{f=\tbinom{\fe}{\fin}\in
\WzpRpCl\times\WzpneCm
\;\bigg|\text{ $f$ satisfies \eqref{eq:BC-gen}}
\biggr\},
\end{aligned}
\end{equation}
generates a cosine family on $X=\LpRpCl\times\LpneCm$ with phase space $V\times X$ for 
\[
V:=\Bigl\{\tbinom{\fe}{\fin}\in\WepRpCl\times\WepneCm:\Vne\fe(0)+\Vni\fin(0)+\Vei\fin(1)=0\Bigr\}.
\]
\end{cor}

\begin{proof}
In order to fit this setting into our general framework let $Y_0:=\CC^{k_0}\times\{0\}^{k_1}\subseteq\CC^{\ell+2m}$ and $Y_1:=\{0\}^{k_0}\times\CC^{k_1}\subseteq\CC^{\ell+2m}$
and define $\Phi_j\in\sL(\CnRpCl\times\CneCm,Y_j)$, $j=0,1$ by
\begin{equation}\label{eq:Phi-delta01}
\Phin:=
\begin{pmatrix}
\Vne\cdot\delta_0&\Vni\cdot\delta_0+\Vei\cdot\delta_1\\
0&0
\end{pmatrix}
,
\quad
\Phie:=
\begin{pmatrix}
0&0\\
\Wne\cdot\delta_0&\Wni\cdot\delta_0-\Wei\cdot\delta_1
\end{pmatrix}.
\end{equation}
Our next aim is to rewrite the boundary conditions \eqref{eq:BC-gen} as
\begin{equation*}
\Phin\tbinom{\fe}{\fin}=0,\quad
\Phie\left(\tbinom{\fe}{\fin}'+ B\tbinom{\fe}{\fin}\right)=0
\end{equation*}
for a suitable operator $B\in\sL(X)$ leaving $\rW^{1,p}$ invariant.
To this end first note that by \eqref{eq:cond-det} there exist matrices $\Rne\in\rM_{\ell \times  k_1}(\CC)$, $\Rni,\Rei\in\rM_{m\times k_1}(\CC)$ such that
\begin{equation}\label{eq:wrid}
\Wneb\cdot\Rne-\Weib\cdot \Rei+\Wnib\cdot \Rni
=Id_{\CC^{k_1}}.
\end{equation}
Denote by $\R\in\sL(\rL^p(\RRp,\CC^{k_1}),\rL^p([0,1],\CC^{k_1}))$ the restriction operator, i.e. $\R f:=f|_{[0,1]}$ for $f\in\rL^p(\RRp,\CC^{k_1})$. Moreover,
let $E\in\sL(\rL^p([0,1],\CC^{k_1}),\rL^p(\RRp,\CC^{k_1}))$ be an extension operator such that $\R Eg=g$ for all $g\in\rL^p([0,1],\CC^{k_1})$ and  $E(\rW^{1,p}([0,1],\CC^{k_1}))\subset\rW^{1,p}(\RRp,\CC^{k_1})$. Now put 
\begin{equation}\label{eq:def-B1}
B:=
\begin{pmatrix}
\ce(0)^{-1}\cdot\Rne\cdot\Be&
\ce(0)^{-1}\cdot\Rne\cdot E\Bi\\[2pt]
C\cdot\R\Be&
C\cdot\Bi
\end{pmatrix}
\in\sL(\Xe\times\Xin),
\end{equation}
where 
\begin{equation*}
C:=\bigl((\eins-\sbb)\cdot \ci(0)^{-1}\Rni+\sbb\cdot \ci(1)^{-1}\Rei\cdot\psi\bigr)
\in\sL\bigl(\rL^p([0,1],\CC^{k_1}),\LpneCm\bigr)
\end{equation*}
for $\eins(s)=1$, $\sbb(s)=s$, $s\in[0,1]$ and $\psi(g):=g(1-\p)$ for $g\in\rL^p([0,1],\CC^{k_1})$. 
Then we have $B(\WepneCm\times\WepRpCl)\subseteq\WepneCm\times\WepRpCl$ and a simple computation using \eqref{eq:wrid} shows that for $f=\tbinom{\fe}{\fin}\in\WzpRpCl\times\WzpneCm$ 
\[
f\text{ satisfies \eqref{eq:BC-gen}}
\quad\iff\quad
\Phin f=0,\ \Phie(f'+B f)=0.
\]
Moreover, for the operators $Q_t,R_t,S_t$ defined in \eqref{eq:def-Rt-Sr} we have
\begin{equation}\label{eq:Rt-St-01}
\begin{aligned}
\delta_0\,\Qt u=u(t),\qquad
&\delta_0\,\Rt v=v(t),
&\quad&\delta_1\,\Rt v=0,\\
&\delta_0\,\St v=0,
&\quad&\delta_1\,\St v=v(t),
\end{aligned}
\end{equation}
where $u\in\rW^{1,p}([0,\tn],\CC^{\ell})$, $v\in\rW^{1,p}([0,\tn],\CC^m)$ 
and $0<\tn\le\min\{\varphi_1(1),\ldots,\varphi_m(1)\}$. Note also that
\begin{equation*}
\delta_0\,\Jp =  \delta_0\,\Jpb = \delta_0\quad\text{and}\quad \delta_1\,\Jpb = \delta_1. 
\end{equation*}
Using all this we compute the operator $\sRtn$ given in \eqref{eq:def-sR} as
\begin{equation}
\sRtn=
\begin{pmatrix}
-\Vne&\Vei&-\Vni\\[2pt]
\Wneb&-\Weib&\Wnib
\end{pmatrix}
\cdot
\diag\bigl(\qe(0),\qi(1),\qi(0)\bigr)
\in\sL\bigl(\rL^p\bigl([0,\tn],\CC^{\ell+2m}\bigr)\bigr).
\end{equation}
Since the matrix $\diag(\qe(0),\qi(1),\qi(0))\in\rM_{\ell+2m}(\CC)$ is always invertible,
the assertion follows from \autoref{thm:gen-Lp}.
\end{proof}

We give some possible choices for the operators $\Be,\Bi$ appearing in \autoref{cor:gen-V-W}.

\begin{exa} \label{ex-matrices}
\makeatletter
\hyper@anchor{\@currentHref}%
\makeatother
\begin{enumerate}[(i), wide=12pt]
\item For matrices $\Une\in\rM_{k_1 \times \ell}(\CC)$, $\Uni,\Uei\in\rM_{k_1\times m}(\CC)$ define the operators
\begin{align*}
\Be:&=\Une\in\sL\bigl(\LpRpCl,\rL^{p}(\RRp,\CC^{k_1})\bigr),\\
\Bi:&=\Uni+\Uei\cdot\psi\in\sL\bigl(\LpneCm,\rL^{p}([0,1],\CC^{k_1})\bigr),
\end{align*}
where $\psi f(\p):=f(1-\p)$ for $f\in\Xin$.
Then $\Be,\Bi$ map $\rW^{1,p}$ into $\rW^{1,p}$ and  for  $f=\binom{\fe}{\fin}\in\WzpRpCl\times\WzpneCm$ the second condition in \eqref{eq:BC-gen}
gives the mixed boundary condition
\begin{equation}\label{eq:BC-mixed}
\Wne(\fe)'(0)+\Wni(\fin)'(0)-\Wei(\fin)'(1)+\Une\,\fe(0)+\Uni\,\fin(0)+\Uei\,\fin(1)=0.
\end{equation}
In particular, this covers the boundary conditions considered in \cite{KPS:08,HKS:15}.

\smallbreak
\item For arbitrary operators $\Te\in\sL(\LpRpCl,\CC^{k_1})$ and $\Ti\in\sL(\LpneCm,\CC^{k_1})$ define
\begin{alignat*}{3}
&\Be\in\sL(\Xe,\rL^{p}(\RRp,\CC^{k_1})),&&
\quad (\Be\fe)(s)=e^{-s}\cdot \Te\fe,\\
&\Bi\in\sL(\Xin,\rL^{p}([0,1],\CC^{k_1})),&&
\quad (\Bi\fin)(s)\equiv \Ti\fin.
\end{alignat*}
Then again $\Be,\Bi$ map $\rW^{1,p}$ into $\rW^{1,p}$ and  for $f=\binom{\fe}{\fin}\in\WzpRpCl\times\WzpneCm$ the second boundary condition in \eqref{eq:BC-gen} reduces to
\begin{equation*}
\Wne(\fe)'(0)+\Wni(\fin)'(0)-\Wei(\fin)'(1)+\Te\fe+\Ti\fin=0.
\end{equation*}
Note that by choosing operators $\Te,\Ti$ properly (e.g.~as an integral)  we thus obtain various non-local boundary conditions. 

\item We can also combine the two examples above and obtain the second condition in \eqref{eq:BC-gen}
of the form
\begin{equation}\label{eq:BC-opT}
\Wne(\fe)'(0)+\Wni(\fin)'(0)-\Wei(\fin)'(1)+\Une\,\fe(0)+\Uni\,\fin(0)+\Uei\,\fin(1)+\Te\fe+\Ti\fin=0.
\end{equation}
\end{enumerate}
\end{exa}

We now show some applications of \autoref{cor:gen-V-W}, first  to simple scalar examples.

\begin{exa}\label{ex-scalar-Dirichlet}
We consider the second derivative $G_p$ and $G_D$ with periodic- and Dirichlet boundary conditions, respectively, on $[0,1]$, that is $G_p,\, G_D\subset\tddss$ on $X:=\rL^p[0,1]$ with domains
\begin{equation*}
\begin{aligned}[3]
D(G_p)&:=\bigl\{f\in \rW^{2,p}[0,1]\mid f(0)=f(1) \text{ and } f'(0) = f' (1) 
\bigr\},\\
D(G_D)&:=\bigl\{f\in \rW^{2,p}[0,1] \mid f(0)=f(1) =0
\bigr\}.
\end{aligned}
\end{equation*}
In order to write these boundary conditions as in \eqref{eq:BC-gen} we choose $\ell=0$\footnote{As before, in the case of empty external part ($\ell=0$) we shorten the notation and omit the superscript  `$i$' for the boundary matrices.} and $m=1$. Moreover, in case of $G_p$  we take  $k_0=k_1=1$ and scalars $V_0=1$, $V_1=-1$, $W_0=W_1=1$. Then the determinant condition \eqref{eq:cond-det} is fulfilled, hence $G_p$ generates a cosine family.  To handle $G_D$ one might be tempted to choose again $k_0=k_1=1$.  Then the first boundary condition $f(0)=0$ can be implemented by choosing $V_0=1$, $V_1=0$ while the second condition $f(1)=0$ follows from \eqref{eq:BC-mixed} if we take $W_0=W_1=U_0=0$ and $U_1=1$. However, by doing so \eqref{eq:cond-det} is not fulfilled nevertheless it is well-known that $G_D$ generates a cosine family. At a first glance, \autoref{cor:gen-V-W} fails in this simple example, thus only gives a sufficient but in general not necessary generation criterion.

However, as pointed out earlier, the matrices $W_0, W_1$ should be used to implement $k_1$ \emph{linear independent} conditions regarding the derivatives at the endpoints. In case of $G_D$ this means that we have to choose $k_0=2$, $k_1=0$ and the two boundary matrices
$V_0= \binom10$, $V_1=\binom01$. For this choice \eqref{eq:cond-det} is fulfilled, yielding the desired generation result. 

We leave it to the reader to check that also problems with Neumann- or mixed boundary conditions on an interval can be handled in the same way.
\end{exa}

Next we consider an example 
showing that the generator property of $G$ not only depends on the matrices $V_0,V_1,W_0$ and $W_1$ which determine the  domain $D(G)$ in \eqref{eq:def-G-delta} but also on the values of the diffusion coefficients $\mu(s)=\sqrt{\lambda(s)}$ for $s=0,1$, appearing in the definitions of $\bar{W_0}$ and $\bar{W_1}$. 

\begin{exa}
For some Lipschitz continuous, positive function $a(\p):[0,1]\to(0,+\infty)$ consider on $X:=\rL^p[0,1]$, $1\le p<+\infty$, the operator
\begin{equation}\label{eq:def-G-expl}
\begin{aligned}
G&:=a(\p)\cdot\tddss,\\
D(G)&:=\bigl\{f\in\rW^{2,p}[0,1] \mid f(0)+f(1)=0,\;f'(0)-f'(1)=0\bigr\}.
\end{aligned}
\end{equation}
Then $G$ is given by \eqref{eq:def-G-delta} for $\ell=0$, $m=1=k_0=k_1$,  and $V_0=W_0=W_1=V_1=1$. 
This gives for $c(\p):=\sqrt{a(\p)}$ the condition
\begin{equation*}
\det
\begin{psmallmatrix}
V_1&V_0\\ \bar{W_1}&\bar{W_0}
\end{psmallmatrix}
=c(0)^{-1}-c(1)^{-1}\ne0
\quad\iff\quad
a(0)\ne a(1).
\end{equation*}
Hence, by the previous result $G$ in \eqref{eq:def-G-expl} generates a cosine family if $a(0)\ne a(1)$.

\smallbreak
We note that in case $a(0)=a(1)$ the operator $G$ never generates a cosine family or even  an analytic semigroup. To prove this assertion, we denote the operator obtained by \eqref{eq:def-G-expl} for $a(\p)\equiv1$ by $G_1$. Then for each $\lambda\in\CC$ we have $e_\lambda\in\ker(\lambda-G_1)$ where
\[
e_\lambda(s):=e^{\sqrt{\lambda}\cdot s}+e^{\sqrt{\lambda}\cdot(1-s)},\ s\in[0,1].
\]
Hence, $\sigma(G_1)=\CC$ implying that $G_1$ cannot be a generator. 

\smallbreak
For Lipschitz continuous, positive $a(\p)\in\rC[0,1]$ one can use the similarity transformation induced by $\Jpb\in\sL(X)$, $\Jpb f:=f\circ\phib$ to show that in case $a(0)=a(1)$ the operators $G$ and $G_1$ are similar up to the perturbation $P:=\frac{a'}2\cdot\dds$. Since $P$ is relatively bounded with bound $0$, this implies the claim.

\smallbreak
Summing up, in this example \autoref{cor:gen-V-W} gives an optimal result which demonstrates the sharpness of the underlying perturbation argument from \autoref{sec:dom-pert}.
\end{exa}

We continue with an example on a simple star-shaped non-compact metric graph, cf.  \autoref{fig:star3}. Further applications to general metric graphs are presented in  \autoref{sec:W&D}.

\begin{figure}[hbt]  
   \centering
   \begin{tikzpicture}[-,auto,node distance=3cm,
  thick,main node/.style={circle,fill=white,draw}]

  \node[main node] (1) at (0,0) {$0$};
  \node (2) at (5,0) {};
  \node[main node] (3) at (-1.5,1.5) {$1$};
  \node[main node] (4) at (-1.5,-1.5) {$1$};

\tikzset{edge/.style = {->,> = latex'}}
 \draw[edge] (1) edge node {$\me_1^e$}  (2);
  \path[every node/.style={font=\sffamily}]
     (1) 
          edge node {$\me_1^i$} (3)
          edge node {$\me_2^i$} (4);
\end{tikzpicture}
 \caption{Star-shaped graph from \autoref{ex:star3}.}  
 \label{fig:star3}
\end{figure}

\begin{exa}\label{ex:star3}
We consider a diffusion process described by $G\subset\ddss$ along the edges of the non-compact star graph presented in \autoref{fig:star3}.
The two compact edges $\me_1^i,\me_2^i$ are parametrized as $[0,1]$, with 0 in the common vertex,  while $\me_1^e= \RRp$.
In the central vertex we  assume continuity and an additional boundary condition for the derivatives, i.e.,
\begin{equation*}\label{eq:star3bc}
\begin{aligned}
&f_1^e(0)=f_1^i(0)=f_2^i(0), \\
&\alpha\cdot (f^e_1)'(0) +  \beta\cdot (f^i_1)'(0) + \gamma\cdot (f^i_2)'(0) = 0,
\end{aligned}
\end{equation*}
for some $\alpha,\beta,\gamma\in\CC$, 
while at the remaining endpoints we set the following  Neumann- and Robin condition, respectively, that is
\begin{equation*}
 (f_1^i)'(1)=0,\quad \delta\cdot f_2^i(1)+ \varepsilon\cdot(f_2^i)'(1) = 0
\end{equation*}
for some $\delta,\varepsilon\in\CC$.
Then $\ell = 1$, $m=2$ and if $\varepsilon\ne0$ we choose  $k_0=2$, $k_1=3$ and boundary matrices
\begin{alignat*}{5}
&\Vne=\begin{psmallmatrix}
0\\1
\end{psmallmatrix},
\quad
&&\Vni=\begin{psmallmatrix}
1&-1 \\0&-1
\end{psmallmatrix},
\quad
&&\Vei=\begin{psmallmatrix}
0&0 \\0&0
\end{psmallmatrix},
\\
&\Wne=\begin{psmallmatrix}
0\\\alpha\\0 
\end{psmallmatrix},
\quad
&& 
\Wni=\begin{psmallmatrix}
0&0\\
\beta&\gamma\\
0&0
\end{psmallmatrix},
\quad
&&
\Wei=\begin{psmallmatrix}
1&0\\0&0 \\0&\varepsilon
\end{psmallmatrix},
\\
&\Une= \begin{psmallmatrix}
0\\
0\\
0
\end{psmallmatrix},
\quad
&&\Uni= \begin{psmallmatrix}
0&0\\
0&0\\
0&0
\end{psmallmatrix},
&&\Uei= \begin{psmallmatrix}
0&0\\
0&0\\
0&-\delta
\end{psmallmatrix}.
\end{alignat*}
Then the determinant in \eqref{eq:cond-det} equals $\varepsilon\cdot(\alpha+\beta + \gamma)$. 
In case $\varepsilon=0$  the boundary conditions are essentially different and
in order to apply \autoref{cor:gen-V-W} one has to take
$k_0 = 3$, $k_1=2$. For accordingly modified boundary matrices this gives determinant $\delta\cdot(\alpha+\beta + \gamma)$.
Hence, if $(|\delta|+|\varepsilon|)\cdot(\alpha+\beta + \gamma)\ne0$, by \autoref{cor:gen-V-W} the problem is well-posed while for $\varepsilon = \delta=0$ it is clearly under-determined.
\end{exa}

\subsection{Boundary conditions via ``boundary spaces''}

We consider another way to impose conditions at the end points using two ``boundary spaces'' $Y_0,Y_1\subset\CC^{\ell+2m}$ and two operators
$\Be\in\sL(\Xe,\rL^p(\RRp,\CC^{\ell+2m}))$, $\Bi\in\sL(\Xin,\rL^p([0,1],\CC^{\ell+2m}))$ satisfying
\begin{equation}\label{eq:inv-sB}
\begin{aligned}
&\Be\,\WepRpCl\subseteq\rW^{1,p}\bigl(\RRp,\CC^{\ell+2m}\bigr),\\
&\Bi\,\WepneCm\subseteq\rW^{1,p}\bigl([0,1],\CC^{\ell+2m}\bigr).
\end{aligned}
\end{equation}
Then for $f=\binom{\fe}{\fin}\in\WzpRpCl\times\WzpneCm$ we consider the boundary conditions
\begin{equation}\label{eq:bc-Y}
\begin{psmallmatrix}
\fe(0)\\ \fin(0)\\ \fin(1)
\end{psmallmatrix}
\in Y_1,\quad
\begin{psmallmatrix}
\phantom{-}\ce(0)\cdot(\fe)'(0)\\ \phantom{-}\ci(0)\cdot(\fin)'(0)\\ -\ci(1)\cdot(\fin)'(1)
\end{psmallmatrix}
+(\Be\fe)(0)+(\Bi\fin)(0)\in Y_0.
\end{equation}
Applying \autoref{thm:gen-Lp}  
to this setting yields the following.

\begin{cor}\label{cor:net-Y}
Let $a(\p,\p)$ and $c(\p,\p)$ be given by \eqref{eq:rep-a(s)} and \eqref{eq:c(s)}, respectively. If for subspaces $Y_0,Y_1\subseteq\CC^{\ell+2m}$ we have
\begin{equation}\label{eq:intersec-0}
Y_0\oplus Y_1=\CC^{\ell+2m},
\end{equation}
then for every $\Be\in\sL(\Xe,\rL^p(\RRp,\CC^{\ell+2m}))$, $\Bi\in\sL(\Xin,\rL^p([0,1],\CC^{\ell+2m}))$ satisfying \eqref{eq:inv-sB}  the operator
\begin{equation}\label{eq:def-G-Y}
\begin{aligned}
G&:=a(\p,\p)\cdot\tddss,\\
D(G)&:=\Bigl\{f=\tbinom{\fe}{\fin}\in\WzpRpCl\times\WzpneCm:f\text{ satisfies \eqref{eq:bc-Y}}\Bigr\},
\end{aligned}
\end{equation}
generates a cosine family on $X=\LpRpCl\times\LpneCm$ with phase space $V\times X$ for 
\[
V:=\l\{f\in\WepneCm:\bigl(\fe(0),\fin(0),\fin(1)\bigr)^\top\in Y_1\r\}.
\]
\end{cor}

\begin{proof} Using that for $I=[0,1]$ or $\RRp$ we have $\rL^p(I,\CC^{\ell+2m})=\rL^p(I,\CC^{\ell})\times\rL^p(I,\CC^m)\times\rL^p(I,\CC^m)$ we decompose $\Be$ and $\Bi$ accordingly, i.e., we write
\begin{align*}
&\Be=
\begin{psmallmatrix}
\Be_1\\ \Be_2\\ \Be_3
\end{psmallmatrix}
:\LpRpCl\to\LpRpCl\times\rL^p(\RRp,\CC^m)\times\rL^p(\RRp,\CC^m),\\
&\Bi=
\begin{psmallmatrix}
\Bi_1\\ \Bi_2\\ \Bi_3
\end{psmallmatrix}
:\LpneCm\to\rL^p([0,1],\CC^{\ell})\times\LpneCm\times\LpneCm.
\end{align*}
Denote by $\R\in\sL(\rL^p(\RRp,\CC^{m}),\rL^p([0,1],\CC^{m}))$ the restriction operator, i.e. $\R f:=f|_{[0,1]}$ for $f\in\rL^p(\RRp,\CC^{m})$. Moreover,
let $E\in\sL(\rL^p([0,1],\CC^{m}),\rL^p(\RRp,\CC^{m}))$ be an extension operator, i.e. $\R Eg=g$ for all $g\in\rL^p([0,1],\CC^{m})$, such that $E\,\rW^{1,p}([0,1],\CC^{m})\subset\rW^{1,p}(\RRp,\CC^{m})$.
Finally, we consider the projection $P\in\sL(\CC^{\ell+2m})$ associated to the representation \eqref{eq:intersec-0}, that is $\ker(P)=\rg(Id-P)=Y_0$ and $\rg(P)=\ker(Id-P)=Y_1$. Then $f=\binom{\fe}{\fin}\in\WzpRpCl\times\WzpneCm$ satisfies \eqref{eq:bc-Y} if and only if $\Phin f=0$ and $\Phie(f'+Bf)=0$ for
\begin{align*}
\Phin:=&(Id-P)\cdot
\begin{psmallmatrix}
\delta_0&0\\0&\delta_0\\0&\delta_1
\end{psmallmatrix}
\in\sL\bigl(\CnRpCl\times\CneCm,Y_0\bigr),\\
\Phie:=&P\cdot
\begin{psmallmatrix}
\phantom{-}\ce(0)\cdot\delta_0&0\\0&\phantom{-}\ci(0)\cdot\delta_0\\0&-\ci(1)\cdot\delta_1
\end{psmallmatrix}
\in\sL\bigl(\CnRpCl\times\CneCm,Y_1\bigr),\\
B:=&
\begin{psmallmatrix}
\ce(0)^{-1}\cdot\Be_1& \ce(0)^{-1}\cdot E\Bi_1\\[2pt]
(\eins-\sbb)\cdot\ci(0)^{-1}\cdot\R\Be_2-\sbb\cdot\ci(1)^{-1}\cdot\psi\cdot\R\Be_3
&\ \,(\eins-\sbb)\cdot\ci(0)^{-1}\cdot\Bi_2-\sbb\cdot\ci(1)^{-1}\cdot\psi\cdot\Bi_3
\end{psmallmatrix}
\in\sL(\Xe\times\Xin),
\end{align*}
where $\sbb(s)=s$, $\eins(s)=1$ for $s\in[0,1]$ and $\psi(h):=h(1-\p)$ for $h\in\Xin$. Note that $B$ leaves $\WepRpCl\times\WepneCm$ invariant, hence \autoref{asu:s-asu-Lp} is satisfied. Now choose $\tn:=\min\{\varphi_1(1),\ldots,\varphi_m(1)\}>0$.
Then a simple computation using \eqref{eq:Rt-St-01} and
\[
\Phieb=\Phie\cdot c(\p,\p)^{-1}=
(\Phiebe,\Phiebi)
=P\cdot
\begin{psmallmatrix}
\phantom{-}\delta_0&0\\0&\phantom{-}\delta_0\\0&-\delta_1
\end{psmallmatrix}
\in\sL\bigl(\CnRpCl\times\CneCm,Y_1\bigr)
\]
yields that $\sRtn$ in \eqref{eq:def-sR} is constant and given by
\[
\sRtn=
\begin{psmallmatrix}
\qe(0)&0&0\\
0&0&\qi(0)\\ 0&-\qi(1)&0
\end{psmallmatrix}
\in\sL\bigl(\rL^p\bigl([0,\tn],\CC^{\ell+2m}\bigr)\bigr).
\]
Hence, $\sRtn$ is invertible and \autoref{thm:gen-Lp} implies the claim.
\end{proof}

We give two possible choices for the operators $\Be$, $\Bi$ appearing in the boundary condition \eqref{eq:bc-Y}.

\begin{exa} 
\begin{enumerate}[(i), wide=12pt]
\item 
For matrices $\Une\in\rM_{(\ell+2m) \times \ell}(\CC)$, $\Uni,\Uei\in\rM_{(\ell+2m)\times m}(\CC)$ define 
\begin{align*}
&\Be:=\Une\in\sL\bigl(\Xe,\rL^p(\RRp,\CC^{\ell+2m})\bigr),\\
&\Bi:=\Uni+\Uei\cdot\psi\in\sL\bigl(\Xin,\rL^p(\RRp,\CC^{\ell+2m})\bigr).
\end{align*}
Then $\Be,\Bi$ satisfy \eqref{eq:inv-sB} and  for $f=\binom{\fe}{\fin}\in\WzpRpCl\times\WzpneCm$ the second boundary condition in \eqref{eq:bc-Y} simplifies to
\begin{equation*}
\begin{psmallmatrix}
\phantom{-}\ce(0)\cdot(\fe)'(0)\\ \phantom{-}\ci(0)\cdot(\fin)'(0)\\ -\ci(1)\cdot(\fin)'(1)
\end{psmallmatrix}
+\Une\fe(0)+\Uni\fin(0)+\Uei\fin(1)
\in Y_0.
\end{equation*}
This generalizes for example the boundary conditions considered in \cite[Sect.~6.5]{Mug:14}, see also \autoref{exa:BC-Y}.

\item  For operators $\Te\in\sL(\LpRpCl,\CC^{\ell+2m})$ and $\Ti\in\sL(\LpneCm,\CC^{\ell+2m})$ define
\begin{alignat*}{3}
&\Be\in\sL\bigl(\Xe,\rL^{p}(\RRp,\CC^{\ell+2m})\bigr),&&
\quad (\Be\fe)(s)=e^{-s}\cdot \Te\fe,\\
&\Bi\in\sL\bigl(\Xin,\rL^{p}([0,1],\CC^{\ell+2m})\bigr),&&
\quad (\Bi\fin)(s)\equiv \Ti\fin.
\end{alignat*}
Then again $\Be,\Bi$ satisfy \eqref{eq:inv-sB}  and  for $f=\binom{\fe}{\fin}\in\WzpRpCl\times\WzpneCm$ the second boundary condition in  \eqref{eq:bc-Y} simplifies to
\begin{equation*}
\begin{psmallmatrix}
\phantom{-}\ce(0)\cdot(\fe)'(0)\\ \phantom{-}\ci(0)\cdot(\fin)'(0)\\ -\ci(1)\cdot(\fin)'(1)
\end{psmallmatrix}
+\Te\fe+\Ti\fin
\in Y_0.
\end{equation*}
\end{enumerate}
\end{exa}

\section{Applications to waves and diffusion on metric graphs}
\label{sec:W&D}

\subsection{Introduction}

In this section we use our abstract results to show the well-posedness of wave- and diffusion equations on networks. That is, we study first and second order abstract initial-boundary value problems of the form \eqref{eq:acp}. The structure of the graph is encoded in the boundary conditions contained in the domain $D(G)$. 

\smallskip
We consider a finite metric graph (network) with $n$ vertices $\mv_1,\dots, \mv_n$,  $m$ \emph{internal} edges $\mei_1,\dots,\mei_m$, which we parametrize on the unit interval $[0,1]$, and $\ell$ \emph{external} edges $\mee_1,\dots,\mee_{\ell}$, parametrized on the half-line $\RR_+$. A graph without external edges is called \emph{compact}.
The structure of the graph is given by the $n\times m$ \emph{internal
incidence matrices} $\Phi^{i,-}:=(\varphi^{i,-}_{rs})$,  and   $\Phi^{i,+}:=(\varphi^{i,+}_{rs})$, where
\begin{equation*}
 \varphi^{i,-}_{rs} := 
 \begin{cases}
  1, & \text{if } \mei_s(0) = \mv_r,\\
  0, & \text{otherwise},
 \end{cases}
 \qquad 
\text{and}
\qquad 
 \varphi^{i,+}_{rs} := 
 \begin{cases}
  1, & \text{if } \mei_s(1) = \mv_r,\\
  0, & \text{otherwise},
 \end{cases}
\end{equation*}
and the $n\times \ell$ \emph{external incidence matrix} $\Phi^{e,-}:=(\varphi^{e,-}_{rs})$, where
\begin{equation*}
 \varphi^{e,-}_{rs} := 
 \begin{cases}
  1, & \text{if } \mee_s(0) = \mv_r,\\
  0, & \text{otherwise}.
 \end{cases}
\end{equation*}
By using the incidence matrices we obtain the diagonal matrices with in-  and out-degrees of all vertices on the diagonal as
\begin{equation}\label{eq:degrees}
D^{\dagger}:=\Phi^{\dagger} (\Phi^{\dagger})^{\top}, \quad  \dagger\in\bigl\{(i,-),(i,+),(e,-)\bigr\},
\end{equation}
and the joint vertex degree matrix 
\begin{equation}\label{eq:matrixD}
D:= D^{i,-} + D^{i,+}+ D^{e,-}= \diag\bigl(\textrm{deg}(\mv_r)\bigr)_{r=1}^{n}\in\rM_n(\CC).
\end{equation}

\smallskip
The diffusion- and wave equation on a metric graph is defined by considering
on each edge the heat equation
\begin{equation}\label{eq:net-dif}
\begin{aligned}
\frac{d}{dt}\, u^i_j(t,s) &= \ai_j(s)\cdot \frac{d^2}{ds^2}\, u^i_j(t,s),\quad t\ge 0,\  s\in(0,1),\quad j=1,\dots,m ,\\
\frac{d}{dt}\, u^e_k(t,s) &= \ae_k(s)\cdot \frac{d^2}{ds^2}\, u^e_k(t,s),\quad t\ge 0,\ s>0, \quad k=1,\dots,\ell,
\end{aligned}
\end{equation}
or the wave equation
\begin{equation}\label{eq:net-wave}
\begin{aligned}
\frac{d^2}{dt^2}\, u^i_j(t,s) &= \ai_j(s)\cdot \frac{d^2}{ds^2}\, u^i_j(t,s),\quad t\ge 0,\  s\in(0,1),\quad j=1,\dots,m ,\\
\frac{d^2}{dt^2}\, u^e_k(t,s) &= \ae_k(s)\cdot \frac{d^2}{ds^2}\, u^e_k(t,s),\quad t\ge 0,\  s>0, \quad k=1,\dots,\ell,
\end{aligned}
\end{equation}
respectively, for some Lipschitz continuous functions $\aek(\p)\in\rC(\RRp)$, $\aij(\p)\in\rC[0,1]$ for $k=1,\ldots,\ell$, $j=1,\ldots,m$, satisfying
\eqref{eq:aek-beschr}.
Additionally, one needs to impose some transmission conditions in the vertices.  These types of problems for compact graphs were studied for example in \cite{KMS:07}. 

Next we present some types of these transmission conditions and show how our results apply in these examples.

\subsection{Standard conditions}\label{sec:standard}

The most natural assumption for the solutions to either heat or wave equations on a metric graph is continuity in the vertices. 
We say that a function $f=\binom{\fe}{\fin}\in\rC(\RR_+,\CC^{\ell})\times\rC([0,1],\CC^{m})$, defined on the edges of the graph, is \emph{continuous on the graph} if its values at the endpoints of the contiguous edges coincide, i.e., whenever two edges $\me_j$ and $\me_k$ (both internal, both external or mixed)  have a common vertex $\mv$ then for the appropriate functions it holds $f_j(\mv) = f_k(\mv)$. Here, $f_j(\mv):=f_j(s)$ if $\me_j(s)=\mv$ for $s=0$ or $s=1$. 
A direct computation shows that the continuity property of $f$ can be easily expressed using the incidence matrices as
\begin{equation}\label{eq:cont-c}
\exists\, c\in \CC^n \text{ such that } (\Phi^{e,-})^\top c = \fe(0),\; (\Phi^{i,-})^\top c = \fin(0) \text{ and } (\Phi^{i,+})^\top c = \fin(1)
\end{equation}
which is equivalent to 
\begin{equation}\label{eq:cont-cond}
\begin{psmallmatrix}
\fe(0)\\ \fin(0) \\ \fin(1)
\end{psmallmatrix}
\in\rg
\begin{psmallmatrix}
(\Phi^{e,-})^\top \\ (\Phi^{i,-})^\top \\  (\Phi^{i,+})^\top
\end{psmallmatrix}.
\end{equation}

\smallskip
Furthermore, in each of the vertices $\mv_r$, $r=1,\dots,n$, we infer  the standard Kirchhoff (also called Neumann) conditions 
\begin{equation*}
\sum_{\me_j\in\Gamma(\mv_r)} \lambda_j(\mv_r)\cdot \frac{\partial f_j}{\partial s}(\mv_r)=0,
\end{equation*}
where $\Gamma(\mv_r)$ denotes the set of all edges incident to the vertex $\mv_r$ and $\frac{\partial f_j}{\partial s_j}(\mv_r)$ is the normal derivative of $f_j$ computed at the appropriate endpoint of the edge $\me_j$. 
Using incidence matrices we can express this condition more accurately as
\begin{equation*}\label{eq:Kirch}
\sum_{k=1}^{\ell} \varphi^{e,-}_{rk}\cdot \aek(0)\cdot (\fe)'_k(0) +\sum_{j=1}^{m} \varphi^{i,-}_{rj}\cdot \aij(0)\cdot (\fin)'_j(0)  = \sum_{j=1}^{m} \varphi^{i,+}_{rj}\cdot \aij(1) \cdot(\fin)'_j(1).
\end{equation*}
Moreover, letting
\begin{alignat*}{3}
&\ae(s):=\diag\bigl(\aek(s)\bigr)_{k=1}^{\ell}\in\rM_{\ell}(\CC),&\quad&s\in\RRp,\\
&\ai(s):=\diag\bigl(\aij(s)\bigr)_{j=1}^m\in\rM_m(\CC),&\quad&s\in[0,1],
\end{alignat*}
we can rewrite the Kirchhoff condition in matrix form as
\begin{equation}\label{eq:trans-cond}
\Phi^{e,-}\cdot \ae(0)\cdot (\fe)'(0) +\Phi^{i,-}\cdot\ai(0)\cdot (\fin)'(0) =  \Phi^{i,+}\cdot \ai(1)\cdot (\fin)'(1).
\end{equation}
Let $a(\p,\p) := \begin{psmallmatrix}
\ae(\ps)&0\\0&\ai(\ps)
\end{psmallmatrix}$ and
\begin{equation}\label{eq:G-networks}
\begin{aligned}
G&:=a(\p,\p)\cdot\tddss,\\
D(G)&:=\bigl\{f=\tbinom{\fe}{\fin}\in\WzpRpCl\times\WzpneCm:f\text{ satisfies  \eqref{eq:cont-cond} and \eqref{eq:trans-cond}}\bigr\}.
\end{aligned}
\end{equation}
Then the diffusion- and wave equations on a network transform into the abstract Cauchy problems given in \eqref{eq:acp}.
We will see that by \autoref{cor:net-Y}  both problems are well-posed.
\smallskip

First we show how  the boundary conditions in the domain $D(G)$ can be written as \eqref{eq:bc-Y} for spaces $Y_0,Y_1$ satisfying \eqref{eq:intersec-0}. 
To this end we define
\begin{equation}\label{eq:Y1-cont}
Y_1:=\rg\begin{psmallmatrix}
(\Phi^{e,-})^\top \\ (\Phi^{i,-})^\top \\  (\Phi^{i,+})^\top
\end{psmallmatrix}
\quad\text{and}\quad\Be=\Bi:=0.
\end{equation}
Moreover, we note that for 
$\ce(\ps):=\sqrt{\ae(\ps)}$ and $\ci(\ps)):=\sqrt{\ai(\ps)}$,
\begin{align*}
\eqref{eq:trans-cond}
\quad&\iff\quad 
\begin{psmallmatrix}
\phantom{-}\ce(0)(\fe)'(0)\\
\phantom{-}\ci(0)(\fin)'(0)\\
- \ci(1)(\fin)'(1)
\end{psmallmatrix}
\in\ker \left(\Phi^{e,-} \ce(0), \Phi^{i,-}\ci(0),\Phi^{i,+}\ci(1)\right)\\
&\iff\quad
\begin{psmallmatrix}
\phantom{-}\ce(0)(\fe)'(0)\\
\phantom{-}\ci(0)(\fin)'(0)\\
- \ci(1)(\fin)'(1)
\end{psmallmatrix}
\in\rg\begin{psmallmatrix}
\ce(0)(\Phi^{e,-})^\top\\
\ci(0)(\Phi^{i,-})^\top\\
\ci(1)(\Phi^{i,+})^\top
\end{psmallmatrix}^\bot
=:Y_0.
\end{align*}
Hence, $Y_0=C Y_1^\bot$ for the invertible diagonal matrix $C:=\diag(\ce(0)^{-1},\ci(0)^{-1}, \ci(1)^{-1})$.
Next we verify that $Y_0\cap Y_1=\{0\}$. Let $y\in Y_0\cap Y_1$.
Then there exists $z\in Y_1^\bot$ such that  $y=C z$ which gives
\[
0=\l<y,z\r>=\l<C z,z\r>.
\]
Since $C$ is positive definite, we conclude that indeed $z=0=y$. Moreover, we have $$\dim(Y_0)+\dim(Y_1)=\dim(Y_1^\bot)+\dim(Y_1)=2m+\ell.$$ This implies \eqref{eq:intersec-0} and hence  \autoref{cor:net-Y} applies.

\subsection{$\delta$-type conditions}
This condition appears in the literature on quantum graphs, see \cite{BK13}.  It consists of the continuity condition  \eqref{eq:cont-c} and the condition 
\begin{equation*}
\sum_{\me_j\in\Gamma(\mv_r)} \lambda_j(\mv_r)\cdot \frac{\partial f_j}{\partial s}(\mv_r)=\alpha_r\cdot f(\mv_r),
\end{equation*}
in every vertex $\mv_r$, $r=1,\dots, n$.
Here $f(\mv_r)$ denotes the common value of the functions $f_j$ corresponding to the edges  $\me_j\in\Gamma(\mv_r)$ that meet in  vertex $\mv_r$, and $\alpha_r$ are some fixed complex coefficients.  
Again we can  rewrite this using incidence matrices as 
\begin{equation*}\label{eq:delta}
\sum_{k=1}^{\ell} \varphi^{e,-}_{rk}\cdot \aek(0) \cdot(\fe)'_k(0) +\sum_{j=1}^{m} \varphi^{i,-}_{rj}\cdot \aij(0)\cdot (\fin)'_j(0)  - \sum_{j=1}^{m} \varphi^{i,+}_{rj} \cdot\aij(1)\cdot (\fin)'_j(1) = \alpha_r\cdot c_r,
\end{equation*}
$r=1,\dots, n$, where $c=(c_1,\dots,c_n)^{\top}$ is the vector appearing in the continuity condition \eqref{eq:cont-c}.
In order to obtain the appropriate matrix form  first note that by 
\eqref{eq:cont-c}, \eqref{eq:degrees} and \eqref{eq:matrixD} we have
\begin{equation*}
D c =  \Phi^{e,-} \fe(0) + \Phi^{i,-}\fin(0) +  \Phi^{i,+} \fin(1).
\end{equation*}
Let  $L :=\diag(\alpha_r)_{r=1}^{n}\in\rM_n(\CC)$. Since every column of an incidence matrix consists of exactly one nonzero entry corresponding to the appropriate endpoint of an edge, there are $m\times m$ and $\ell\times\ell$ diagonal matrices $\tilde{D}^{\dagger}$,  such that 
\begin{equation*}
LD^{-1} \Phi^{\dagger} = \Phi^{\dagger} \tilde{D}^{\dagger},\quad \dagger\in\bigl\{(i,-),(i,+),(e,-)\bigr\}.
\end{equation*}
Hence we can rewrite $\delta$-type conditions in the matrix form as
\begin{equation}\label{eq:delta-cond}
\begin{aligned}
&\Phi^{e,-}\cdot \ae(0)\cdot (\fe)'(0) + \Phi^{i,-}\cdot\ai(0) \cdot(\fin)'(0)- \Phi^{i,+}\cdot \ai(1) \cdot(\fin)'(1) \\
&\qquad=  \Phi^{e,-}\cdot \tilde{D}^{e,-}\cdot\fe(0) + \Phi^{i,-}\cdot\tilde{D}^{i,-}\cdot\fin(0) +  \Phi^{i,+}\cdot \tilde{D}^{i,+}\cdot\fin(1).
\end{aligned}
\end{equation}
Defining $Y_0$ and $Y_1$ as in \autoref{sec:standard} and the operators
\begin{equation*}
\Be:= - \begin{psmallmatrix}
\ce(0)^{-1}\tilde{D}^{e,-}\\ 0\\ 0
\end{psmallmatrix}
\quad\text{and}\quad
\Bi:= - \begin{psmallmatrix}
0\\ \ci(0)^{-1}\tilde{D}^{i,-}\\ \ci(1)^{-1}\tilde{D}^{i,+}\psi
\end{psmallmatrix},
\end{equation*}
where $\psi(h):=h(1-\p)$, 
our boundary conditions are of the form \eqref{eq:bc-Y} and \autoref{cor:net-Y} applies again.

\subsection{Non-local boundary conditions} We now further generalize the standard boundary conditions, taking the continuity condition \eqref{eq:cont-cond} together with the condition
\begin{equation}\label{eq:nonlocal}
\begin{aligned}
&\Phi^{e,-}\cdot \ae(0)\cdot (\fe)'(0) + \Phi^{i,-}\cdot\ai(0) \cdot(\fin)'(0)- \Phi^{i,+}\cdot \ai(1) \cdot(\fin)'(1) \\
&\qquad=  \Phi^{e,-}\cdot M^{e}\cdot\fe(0) + \Phi^{i,-}\cdot M^{i,-}\cdot\fin(0) +  \Phi^{i,+}\cdot M^{i,+}\cdot\fin(1),
\end{aligned}
\end{equation}
for some matrices $M^e\in M_{\ell}(\mathbb{C})$ and  $M^{i,-},M^{i,+} \in M_{m}(\mathbb{C})$. Note that in this way the Kirchhoff conditions in a vertex are supplemented with a linear combination of values in some other -- even non-adjacent -- vertices. This models, for example, a network, in which some nodes are able to communicate instantly and directly via another network, atop of the one under consideration. 
To treat this case we may again define $Y_0$ and $Y_1$ as in \autoref{sec:standard}, take the boundary operators
\begin{equation*}
\Be:= - \begin{psmallmatrix}
\ce(0)^{-1} M^{e}\\ 0\\ 0
\end{psmallmatrix}
\quad\text{and}\quad
\Bi:= - \begin{psmallmatrix}
0\\ \ci(0)^{-1} M^{i,-}\\ \ci(1)^{-1} M^{i,+}\psi
\end{psmallmatrix},
\end{equation*}
with $\psi(h):=h(1-\p)$ and apply \autoref{cor:net-Y}.

\subsection{Matrix mixed conditions}
Motivated by applications in population dynamics, in \cite{BFN:16a,BFN:16} a diffusion problem on a compact network with the general boundary condition
\begin{equation*}
\tbinom{f'(0)}{f'(1)}= \mathbb{K} \tbinom{f(0)}{f(1)}
\end{equation*}
for a matrix $\mathbb{K}\in\rM_{2m}(\CC)$ is considered.
In this case \autoref{cor:net-Y}  applies directly by choosing $\ell = 0$, $a(\p,\p) \equiv\Id$, $Y_0 := \{0\}$, $Y_1:=\CC^{2m}$ and
\begin{equation*}
\Bi:=-  \begin{psmallmatrix}Id & 0\\ 0& -Id\end{psmallmatrix}\cdot\mathbb{K}\cdot\tbinom{Id}{\psi},
\end{equation*}  
where as usual $\psi f(\p):=f(1-\p)$.

\subsection{Generalized node conditions}\label{exa:BC-Y}
In \cite[Sect.~6.5]{Mug:14} the boundary condition 
\begin{equation}\label{eq:bc-Y1}
\tbinom{f(0)}{f(1)}\in Y,\quad
\tbinom{-\lambda(0)\cdot f'(0)}{\phantom{-}\lambda(1)\cdot f'(1)}+W\tbinom{f(0)}{f(1)}\in\Yl
\end{equation}
appears for compact graphs, where $Y\subseteq\CC^{2m}$ and $W\in\sL(Y)$. We show that also this condition fits in the setting of \autoref{cor:net-Y} for $\ell=0$. To this end we define $Y_1:=Y$, $Y_0:=C\,\Yl$  for
\[
C:=\diag\bigl(\mu(0)^{-1},\mu(1)^{-1}\bigr) \quad\text{and}\quad \Bi:=CW\tbinom{Id}{\psi},
\] 
where $\mu(\p):=\sqrt{\lambda(\p)}$, $\psi f(\p):=f(1-\p)$. Then a simple computation shows that for these choices \eqref{eq:bc-Y1} is equivalent to \eqref{eq:bc-Y}.
Next, the representation $Y_0:=C\, Y_1^\bot$ for positive definite $C$ implies by the same reasoning as at the end of \autoref{sec:standard} condition \eqref{eq:intersec-0}. Hence, \autoref{cor:net-Y} applies to the operator $G=a(\p)\cdot\ddss$ satisfying the boundary conditions \eqref{eq:bc-Y1}. Moreover, this condition  can be easily generalized  to the non-compact metric graphs.

\appendix

\section{}
\label{app:dp}

\subsection{Domain perturbation for generators of $\mathbf{C_0}$-semigroups}
\label{sec:dom-pert}
In this appendix we briefly recall a perturbation result from \cite[Sect.~4.3]{ABE:13} which is our main tool to prove \autoref{thm:gen-Lp} (similar see also \cite{Had:05,HMR:15}).
Moreover, we give an admissibility criterion which significantly simplifies the computation of the so-called controllability- and input-output maps. 
To explain the general setup we consider
\begin{itemize}
\item two Banach spaces $X$ and $\partial X$, called ``state'' and ``boundary'' spaces, respectively;
\item a closed, densely defined ``maximal'' operator\footnote{``maximal'' concerns the size of the domain, e.g., a differential operator without boundary conditions.} $A_m:D(A_m)\subseteq X\to X$;
\item the Banach space $[D(A_m)]:=(D(A_m),\|\p\|_{A_m})$ where $\|f\|_{A_m}:=\|f\|+\|A_mf\|$ is the graph norm;
\item two ``boundary'' operators $L,C\in\sL([D(A_m)],\partial X)$.
\end{itemize}

Then define two restrictions $A,\,G\subset A_m$ by
\begin{align}
D(A):&=\bigl\{f\in D(A_m):Lf=0\bigr\}=\ker(L),\notag\\
D(G):&=\bigl\{f\in D(A_m):Lf=C f\bigr\}=\ker(L-C)\label{eq:def-G}.
\end{align}
Hence, one can consider $G$ with boundary condition $Lf=C f$ as a perturbation of the operator $A$ with abstract ``Dirichlet type'' boundary condition $Lf=0$.
In order to proceed we make the following 

\begin{asu}\label{asu:BP}
\makeatletter
\hyper@anchor{\@currentHref}%
\makeatother
\begin{enumerate}[(i), wide=12pt]
\item The operator $A$ generates a $C_0$-semigroup $\Tt$ on $X$;
\item the boundary operator $L:D(A_m)\to\partial X$ is surjective.
\end{enumerate}
\end{asu}

Under these assumptions the following lemma is shown in \cite[Lem.~1.2]{Gre:87}.

\begin{lem}\label{lem-Gre} Let \autoref{asu:BP} be
satisfied. Then for each $\lambda\in\rho(A)$ the operator $L|_{\ker(\lambda-A_m)}$ is invertible and
$L_\lambda:=(L|_{\ker(\lambda-A_m)})^{-1}:\partial X\to\ker(\lambda-A_m)\subseteq X$
is bounded.
\end{lem}

In what follows, the extrapolated space $X_{-1}$ associated with $A$ 
is the completion of $X$ with respect to the norm
\[\|x\|_{-1}:= \|R(\lambda_0,A) x\|, \quad x \in X,\]
for some fixed $\lambda_0\in\rho(A)$,
$T_{-1}(t)\in\sL(\Xme)$ is the unique bounded extension of the operator $T(t)$ to $X_{-1}$, and $A_{-1}$ is the generator of the extrapolated semigroup $\left(T_{-1}(t)\right)_{t\ge 0}$ with domain $D(A_{-1})=X$. For more details on extrapolated spaces and semigroups we refer to \cite[Sect.~II.5.a]{EN:00}.

Now one can verify that the operator
\[
L_A:=(\lambda-\Ame)L_\lambda\in\sL(\dX,\Xme)
\]
is independent of $\lambda\in\rho(A)$ and that $G=(\Ame+L_A\cdot C)|_X$. Before stating the perturbation result \cite[Cor.~22]{ABE:13}, we note that from the assumptions~(i)--(iii) in \autoref{thm:pert-bc} below it follows that there exists a bounded ``\emph{input-output map}''
$\sF_{\tn}\in\sL(\rL^p([0,t_0],\dX))$ such that
\begin{equation}\label{sFt}
(\sF_{\tn} u)(\p)= C\int_0^\p T_{-1}(\p-s)L_A u(s)\ds
\quad\text{for all }u\in\rW_0^{2,p}([0,t_0],\dX),
\end{equation}
cf. \cite[Rem.~7]{ABE:13}. Now by \cite[Thm.~10]{ABE:13} the following holds.

\begin{thm}\label{thm:pert-bc}
Assume that there exist $1\le p<+\infty$, $\tn>0$ and $M\ge0$ such that
\begin{alignat*}{2}
\text{(i)}\quad&\int_0^{t_0} T_{-1}(t_0-s) L_A u(s)\ds \in X                    &&\text{for all }u\in\rL^p([0,t_0],\dX),\\
\text{(ii)}\quad&\int_0^{t_0}\bigl\|CT(s)x\bigr\|_\dX^p\ds \leq M\cdot\|x\|_X^p&&\text{for all }x\in D(A),\\
\text{(iii)}\quad&\int_0^{\tn}\Bigl\|C\int_0^r T_{-1}(r-s) L_A u(s)\ds\Bigr\|_\dX^p\dr\le M\cdot\|u\|_p^p&&\text{for all }u\in\rW_0^{2,p}([0,t_0],\dX),\\
\text{(iv)}\quad&1\in\rho(\sF_{\tn})\text{, where }\sF_{\tn}\in\sL\bigl(\rL^p([0,t_0],\dX)\bigr)\text{ is given by \eqref{sFt}}.&&
\end{alignat*}
Then $G\subset A_m$ given by \eqref{eq:def-G} generates a $C_0$-semigroup on the Banach space $X$.
\end{thm}

\begin{rem}\label{rem:C-admiss}
If assumption~(ii) in \autoref{thm:pert-bc} is satisfied, then the operator $C$ is called a $p$-admissible \emph{observation operator} for $\Tt$. In this case there exist a unique ``\emph{observability map}'' $\sC_{\tn}\in\sL(X,\rL^p[0,\tn],\dX))$ such that
\begin{equation*}
\sC_{\tn} x=C\cdot T(\p)x
\qquad\text{for all }x\in D(A).
\end{equation*}
\end{rem}

If we put $\Phi:=L-C\in\sL([D(A_m)],\partial X)$ we obtain the following slight modification of \autoref{thm:pert-bc} which fits better our needs in \autoref{sec:Lp-gen}.

\begin{cor}\label{thm:pert-bc-v2}
Assume that there exist $1\le p<+\infty$, $\tn>0$ and $M\ge0$ such that
\begin{alignat*}{2}
\text{(i)}\quad&\int_0^{t_0} T_{-1}(t_0-s) L_A v(s)\ds \in X                    &&\text{for all }v\in\rL^p([0,t_0],\dX),\\
\text{(ii)}\quad&\int_0^{t_0}\bigl\|\Phi\,T(s)x\bigr\|_\dX^p\ds \leq M\cdot\|x\|_X^p&&\text{for all }x\in D(A),\\
\text{(iii)}\quad&\int_0^{\tn}\Bigl\|\Phi\int_0^r T_{-1}(r-s) L_A v(s)\ds\Bigr\|_\dX^p\dr\le M\cdot\|v\|_p^p&&\text{for all }v\in\rW_0^{2,p}([0,t_0],\dX),\\
\text{(iv)}\quad&\sQ_{\tn}\text{ is invertible, where }\sQ_{\tn}\in\sL\bigl(\rL^p([0,t_0],\dX)\bigr)\text{ is given by }&&
\end{alignat*}
\begin{equation*} 
(\sQ_{\tn} v)(\p)=\Phi\int_0^\p T_{-1}(\p-s)L_A v(s)\ds
\quad\text{for all }v\in\rW_0^{2,p}([0,t_0],\dX),
\end{equation*}
Then $G\subset A_m$, $D(G):=\ker(\Phi)$, generates a $C_0$-semigroup on the Banach space $X$.
\end{cor}

\begin{proof}
Since the assumptions~(i) are the same, it suffices to show that the hypotheses (ii)--(iv) imply the corresponding assumptions in \autoref{thm:pert-bc}.

\smallbreak
(ii) This is clear since $LT(s)x=0$ for all $x\in D(A)=\ker(L)$ and $s\ge0$.

(iii) Using integration by parts twice, one sees that for all $v\in\rW_0^{2,p}([0,t_0],\dX)$
\[
L\int_0^r T_{-1}(r-s) L_A v(s)\ds=v(r),
\ r\in[0,\tn]
\]
which implies (iii) in the previous result.

(iv) By the previous point it also follows that $\Id-\sF_{\tn}=\sQ_{\tn}$ which implies the corresponding assumption in \autoref{thm:pert-bc}.
\end{proof}

In \cite{EKKNS:10} we showed two versions of variation of parameters formula for the solutions to boundary control problems. By using them we obtain the following equivalence which is quite helpful to verify the first assumption in the previous two results.

\begin{lem}\label{lem:admiss-LA}
For $p\in[1,\infty)$ the following are equivalent.
\begin{enumerate}
\item Assumption~(i) in \autoref{thm:pert-bc} and \autoref{thm:pert-bc-v2} is satisfied.
\item There exists $\tn>0$ and a strongly continuous family $\Bttn\subset\sL(\rL^p([0,\tn],\dX),X)$ such that for every $u\in\rW^{2,p}_0([0,\tn],\dX)$ the function
\begin{equation*}
x:[0,\tn]\to X,\quad x(t):=\sBt u
\end{equation*}
is a classical solution of the boundary control problem
\begin{equation}
\label{BCP}
\begin{cases}
\dot x(t)=A_m x(t),&0\le t\le\tn,\\
L x(t)=u(t),&0\le t\le\tn,\\
x(0)=0.
\end{cases}
\end{equation}
\end{enumerate}
In this case for $t\in(0,\tn]$ the operator $\sBt$ coincides with the ``\emph{controllability map}'', i.e.,
\begin{equation}\label{eq:def-cont-map}
\sBt u=\int_0^{t} T_{-1}(t-s) L_A u(s)\ds
\qquad\text{for }u\in\rL^p[0,\tn],\dX).
\end{equation}
\end{lem}

\begin{proof}
(a)$\Rightarrow$(b) By assumption \eqref{eq:def-cont-map}
defines a bounded operator from $\rL^p([0,\tn],\dX)$ to $\Xme$ satisfying $\rg(\sBt)\subset X$. By the closed graph theorem and \cite[Cor.~3.16]{BE:13} this implies that $\Bttn\subset\sL(\rL^p([0,\tn],\dX),X)$ is strongly continuous. Finally, by \cite[Prop.~2.8]{EKKNS:10} for $u\in\rW^{2,p}_0([0,\tn],\dX)$ the function $x(t):=\sB_tu$ gives the unique classical solution of \eqref{BCP}.

\smallbreak
(b)$\Rightarrow$(a) Define $\tilde\sB_{t_0}\in\sL(\rL^p([0,\tn],\dX),\Xme)$ by the right-hand-side of \eqref{eq:def-cont-map} for $t=\tn$. Then by \cite[Prop.~2.7]{EKKNS:10} we have
\[
\tilde\sB_{t_0}\big|_{\rW^{2,p}_0([0,\tn],\dX)}=\sB_{t_0}\big|_{\rW^{2,p}_0([0,\tn],\dX)}.
\]
Since $\rW^{2,p}_0([0,\tn],\dX)\subset\rL^p([0,\tn],\dX)$ is dense this implies $\tilde\sB_{t_0}=\sB_{t_0}\in\sL(\rL^p([0,\tn],\dX),X)$. Hence, $\rg(\tilde\sB_{t_0})\subseteq X$, i.e., (a) is satisfied.
\end{proof}

For more details and examples regarding the above perturbation results we refer to \cite{ABE:13, ABE:15, BE:13}. 

\subsection{Two auxiliary results}

We state and prove two results concerning the inverse and derivative of matrix-valued functions. 

\begin{lem}\label{lem:SR-Lip.cont} Let $I\subseteq\RR$ be an interval and $d(\p):I\to\rM_n(\CC)$ be Lipschitz continuous and bounded, such that $\sigma(d(s))\subset(0,+\infty)$ for all $s\in I$.
If $I$ is not compact assume in addition that there exists $\eps>0$ such that $\sigma(d(s))\subseteq[\eps,\frac1\eps]$ for all $s\in I$.
Then $d^{-1}(\p):I\to\rM_n(\CC)$ is Lipschitz continuous and bounded as well.
\end{lem}

\begin{proof} 
By assumption or by compactness of $I$ there exists $\eps>0$ such that $\sigma(d(s))\subset[\eps,\frac1\eps]$ for all $s\in I$. This implies $\det(d(s)^{-1})\in[\eps^n,\frac1{\eps^n}]$  for all $s\in I$, hence by the inversion formula for matrices there exists $M>0$ such that $\|d(s)^{-1}\|\le M$ for all $s\in I$. Moreover,
\begin{align*}
\bigl\|d^{-1}(s)-d^{-1}(r)\bigr\|\le
\bigl\|d^{-1}(r)\bigr\|\cdot\bigl\|d(r)-d(s)\bigr\|\cdot\bigl\|d^{-1}(s)\bigr\|
\le M^2\cdot\bigl\|d(r)-d(s)\bigr\|
\end{align*}
for all $s,r\in I$, i.e. $d^{-1}(\p)$ is bounded and Lipschitz continuous as claimed.
\end{proof}

\begin{cor}\label{cor:reg-prod}
Let $d(\p)$ as in \autoref{lem:SR-Lip.cont}.
Then for $f:I\to\CC^n$ we have
\[
f\in\rW^{1,p}(I,\CC^n)
\quad\iff\quad
d\cdot f\in\rW^{1,p}(I,\CC^n)
\]
and in this case $(d\cdot f)'=d'\cdot f+d\cdot f'$.
\end{cor}

\begin{proof} ``$\Rightarrow$'': Since $d(\p)\in\rW^{1,p}_\loc(I,\rM_n(\CC))$, by \cite[Cor.~8.10]{Bre:11} we conclude that $d\cdot f\in\rW^{1,p}_{\loc}(I,\CC^n)$ and $(d\cdot f)'=d'\cdot f+d\cdot f'$. By assumption or by compactness of $I$ there exists $\eps>0$ such that $\sigma(d(s))\subset[\eps,\frac1\eps]$ for all $s\in I$. If $L$ denotes the Lipschitz constant for $d(\p)$ this implies
\begin{align*}
\|d\cdot f\|_{\rW^{1,p}(I,\CC^n)}^p
\le\tfrac1{\eps^p}\cdot\|f\|_p^p+L^p\cdot\|f\|_p^p+\tfrac1{\eps^p}\cdot\|f'\|_p^p<+\infty,
\end{align*}
hence $d\cdot f\in\rW^{1,p}(I,\CC^n)$.
To show ``$\Leftarrow$'' we write $f=d^{-1}\cdot df$ and use \autoref{lem:SR-Lip.cont}.
\end{proof}

\nocite{EN:00}

\subsection*{Acknowledgments}
The authors thank the anonymous referee for many helpful comments and suggestions which helped to improve the readability of this paper and to broaden its scope.
The second author acknowledges financial support from the Slovenian Research Agency, Grant No.~P1-0222.


\end{document}